\documentclass[12pt, onecolumn]{article}

\usepackage[dvipdfmx]{graphicx}
\usepackage[utf8]{inputenc}
\usepackage[T1]{fontenc}
\usepackage[numbers]{natbib}
\usepackage{amsmath}
\usepackage{amsthm}
\usepackage{amssymb}
\usepackage{algorithm}
\usepackage{algorithmic}
\usepackage{ecltree}
\usepackage{enumerate}
\usepackage{xcolor}
\usepackage{comment}
\usepackage{tikz}
\usepackage{wrapfig}
\usepackage{mathrsfs}
\usepackage{mathtools}
\usepackage[margin=25truemm]{geometry}
\mathtoolsset{showonlyrefs=true}

\usepackage{combgames}

\usepackage{bm}
\usepackage{bbm}
\usepackage{float}
\usepackage{cases}
\usepackage{longtable}
\usepackage{diagbox}

\newtheorem{proposition}{Proposition}[section]
\newtheorem{lemma}[proposition]{Lemma}
\newtheorem{theorem}[proposition]{Theorem}
\newtheorem{corollary}[proposition]{Corollary}

\theoremstyle{definition}
\newtheorem{definition}[proposition]{Definition}
\newtheorem{example}[proposition]{Example}
\newtheorem{remark}[proposition]{Remark}

\newcommand{\NN}{{\mathcal{N}}}
\newcommand{\PP}{{\mathcal{P}}}
\newcommand{\LL}{{\mathcal{L}}}
\newcommand{\RR}{{\mathcal{R}}}

\newcommand{\st}{{\ast}}
\newcommand{\up}{{\uparrow}}
\newcommand{\doubleup}{{\Uparrow}}

\newcommand{\seq}{\mathbin{\to}}

\newcommand{\aint}{\mathbb{Z}}
\newcommand{\uint}{\mathbb{Z}_{\geq 0}}
\newcommand{\pint}{\mathbb{Z}_{\geq 1}}

\newcommand{\aw}{\mathrm{aw}}

\newcommand{\wset}{\mathcal{W}}
\newcommand{\sset}[1]{\mathcal{F}({#1})}
\newcommand{\spos}[3]{{#1}({#2},{#3})}
\newcommand{\pos}[2]{{#1}\langle{#2}\rangle}
\newcommand{\ppos}[3]{{#1}\langle{#2},{#3}\rangle}

\newcommand{\eqlab}[2]{\overset{(\mathrm{#1})}{#2}}

\makeatletter

\@addtoreset{equation}{section}
\makeatother

\title{Partizan Serial Nim}
\author{Kengo Hashimoto}
\date{}

\begin{document}
\maketitle

\begin{abstract}
A combinatorial game is a two-player game without hidden information or chance elements.
The main object of combinatorial game theory is to determine the outcome
(i.e., which player has a winning strategy) of a given position in combinatorial games.
\textsc{Nim} is a well-known and fundamental ruleset in combinatorial game theory.
This paper proposes a novel partizan variant of \textsc{Nim} called \textsc{Partizan-Serial-Nim}, defined as follows:
there are $n$ piles of stones indexed by $1, 2, \ldots, n$;
the two players have permutations $\bm{\sigma}^L$ and $\bm{\sigma}^R$ of $(1, 2, \ldots, n)$, respectively;
a move is to remove any positive number of stones from the non-empty pile with the minimum value in the player's permutation;
the player who cannot make a move loses.
This ruleset is a generalization of \textsc{Serial-Nim} and \textsc{Partizan-End-Nim}.
We give an algorithm to compute the outcome of a given position in \textsc{Partizan-Serial-Nim} in $O(n^2)$ time, provided that each arithmetic and comparison operation is performed in $O(1)$ time.
Also, for the case where all non-empty piles have the same number $m$ of stones,
we prove that the outcome does not depend on $m$ for $m \geq 2$ and present an algorithm to compute the outcome in $O(n)$ time.
Further, we prove that the atomic weight of every position in \textsc{Partizan-Serial-Nim} is an integer.
\end{abstract}

\section{Introduction}

A combinatorial game is a two-player game without hidden information or chance elements; e.g., \textsc{Chess}, \textsc{Go}, and \textsc{Checkers}.
In a combinatorial game, two players, conventionally called Left (female) and Right (male), make moves alternately.
If it is guaranteed that a winner is determined within a finite number of moves, then exactly one of the two players has a winning strategy; that is, by continuing to make appropriate moves according to a certain strategy,
the player can win regardless of the opponent's moves. The main object of combinatorial game theory is to determine the \emph{outcome}
(i.e., which player has a winning strategy) of a given position of a combinatorial game.

As a classical result in combinatorial game theory, Bouton's theorem \cite{Bou1901} on \textsc{Nim} is well-known.
\textsc{Nim} is a ruleset played on $n$ piles of stones; a move in \textsc{Nim} is to remove any positive number of stones from any single pile;
the player who cannot make a move loses.

Levine \cite{Lev06} studied a variant of \textsc{Nim} called \textsc{Serial-Nim}, where the $n$ piles are prioritized as integers $1, 2, \ldots, n$, and
a move is to remove any positive number of stones from the non-empty pile with the highest priority (i.e., assigned the minimum integer).
This can be viewed as a sequential compound \cite{Has26, Ste07, SU93} of positions where every component is a position in \textsc{Nim} with a single pile.

For a partizan generalization of \textsc{Serial-Nim}, it is natural to consider the setting where the two players have different priorities on the piles.
This motivates us to propose a new ruleset, \textsc{Partizan-Serial-Nim}, defined as follows.
\begin{quote}
There are $n$ piles of stones. 
Left and Right have permutations $\bm{\sigma}^L$ and $\bm{\sigma}^R$ of $(1, 2, \ldots, n)$, respectively.
The permutation $\bm{\sigma}^L$ (resp.~$\bm{\sigma}^R$) indicates the priority of each pile for Left (resp.~Right), where the smaller the value, the higher the priority.
A move is to remove any positive number of stones from the player's \emph{target pile}, the non-empty pile with the highest priority for the player.
A winner is determined by normal play convention: the player who cannot make a move loses.
\end{quote}

Without loss of generality, we may fix $\bm{\sigma}^L = (1, 2, \ldots, n)$ by changing the indices of the piles if necessary.
Then the priorities of both players are determined only by $\bm{\sigma}^R$, which we write simply as $\bm{\sigma} = (\sigma(1), \sigma(2), \ldots, \sigma(n))$ without the superscripts $L$ and $R$.

\begin{example}
Let $\bm{\sigma} = (10, 1, 3, 9, 2, 5, 7, 6, 8, 4)$, and
suppose the initial position is represented as $\bm{w} = (2, 3, 1, 5, 2, 3, \allowbreak 1, 3, 2, 3)$,
where the $i$-th element $w_i$ of $\bm{w}$ indicates the number of stones in the $i$-th pile.
Let $\bm{\sigma}^{-1} = (\sigma^{-1}(1), \sigma^{-1}(2), \ldots, \allowbreak \sigma^{-1}(n))
= (2, 5, 3, 10, 6, 8, 7, 9, 4, 1)$ be the inverse permutation of $\bm{\sigma}$ (i.e., $\sigma^{-1}(\sigma(i)) = i$).
Let $l(\bm{w})$ (resp.~$r(\bm{w})$) denote the index of Left's (resp.~Right's) target pile.
The following is an example of a possible play where Right plays first.
\begin{itemize}
\item At first, $r(\bm{w}) = 2$ since $\sigma^{-1}(1) = 2$ and $w_2 > 0$. Right removes $2$ stones from the $2$nd pile.
This results in $\bm{w} = (2, 1, 1, 5, 2, 3, 1, 3, 2, 3)$.
\item Then $l(\bm{w}) = 1$ since $w_1 > 0$. Left removes $2$ stones from the $1$st pile. This results in $\bm{w} = (0, 1, 1, 5, 2, 3, 1, 3, 2, 3)$.
\item Then $r(\bm{w}) = 2$ since $\sigma^{-1}(1) = 2$ and $w_2 > 0$. Right removes $1$ stone from the $2$nd pile.
This results in $\bm{w} = (0, 0, 1, 5, 2, 3, 1, 3, 2, 3)$.
\item Then $l(\bm{w}) = 3$ since $w_1 = w_2 = 0$ and $w_3 > 0$. Left removes $1$ stone from the $3$rd pile.
This results in $\bm{w} = (0, 0, 0, 5, 2, 3, 1, 3, 2, 3)$.
\item Then $r(\bm{w}) = 5$ since $\sigma^{-1}(1) = 2, \sigma^{-1}(2) = 5$, $w_2 = 0$, and $w_5 > 0$. Right removes $2$ stones from the $5$th pile. This results in $\bm{w} = (0, 0, 0, 5, 0, 3, 1, 3, 2, 3)$.
\item Then $l(\bm{w}) = 4$ since $w_1 = w_2 = w_3 = 0$ and $w_4 > 0$. Left removes $4$ stones from the $4$th pile.
This results in $\bm{w} = (0, 0, 0, 1, 0, 3, 1, 3, 2, 3)$.
\item Then $r(\bm{w}) = 10$ since $\sigma^{-1}(1) = 2, \sigma^{-1}(2) = 5, \sigma^{-1}(3) = 3, \sigma^{-1}(4) = 10$,
$w_2 = w_5 = w_3 = 0$, and $w_{10} > 0$. Right removes $1$ stone from the $10$th pile. This results in $\bm{w} = (0, 0, 0, 1, 0, 3, 1, 3, 2, 2)$.
\item In this manner, the play continues until all piles are empty, and then the player with the last move wins by normal play convention.
\end{itemize}
\end{example}

Note that \textsc{Serial-Nim} is the particular case of \textsc{Partizan-Serial-Nim} where $\bm{\sigma} = (1, 2, \ldots, n)$.
The particular case $\bm{\sigma} = (n, n-1, \ldots, 1)$ is also known as \textsc{Partizan-End-Nim} \cite{AN01,ANW19,DKW09}.
Albert and Nowakowski \cite{AN01} gave an algorithm to determine the outcome of a given position in \textsc{Partizan-End-Nim} in $O(n^2)$ time provided that each arithmetic and comparison operation is performed in $O(1)$ time.
Duffy et al.~\cite{DKW09} further refined the proofs in \cite{AN01} and generalized them by allowing piles to have an ordinal number of stones.

\subsection{Contributions and Organization}

In this paper, we mainly provide the following three results:
\begin{itemize}
\item in Section \ref{sec:general},
we extend the results in \cite{AN01, DKW09} and provide an algorithm to compute the outcome of a given position in \textsc{Partizan-Serial-Nim} in $O(n^2)$ time as Theorem \ref{thm:outcome};
\item in Section \ref{sec:uniform},
for the \emph{uniform case}, in which the number of stones in each pile is uniformly $m$, we prove that the outcome does not depend on $m$ for $m \geq 2$,
and we present an algorithm to compute the outcome in $O(n)$ time as Theorem \ref{thm:const-outcome};
\item in Section \ref{sec:aw}, we prove that the atomic weight of every position in \textsc{Partizan-Serial-Nim} is an integer as Theorem \ref{thm:aw-integer},
which partially confirms the conjecture on \textsc{Partizan-End-Nim} with ordinal numbers of stones, stated in \cite{DKW09}.
\end{itemize}

Before presenting the above results, 
we introduce the formal notation and definition of the ruleset \textsc{Partizan-Serial-Nim} in Section \ref{sec:ruleset}.

Although this paper focuses on normal play,
our results are also applicable to mis\`{e}re play, in which the player with the last move loses, because
\textsc{Partizan-Serial-Nim} under mis\`{e}re play is easily reduced to one under normal play by taking a sequential compound (cf.~Remark \ref{rem:seq}).
We also remark that the discussion on the uniform case makes sense by our generalization:
when limited to \textsc{Partizan-End-Nim}, the outcome in the uniform case is determined solely by the parity of $n$ because
if $n$ is even, then the mirror-image strategy works, and if $n$ is odd, then removing the whole target pile is a winning move to reduce the position to the case where $n$ is even.

In this paper, we assume that each arithmetic and comparison operation is performed in $O(1)$ time.
For basic notation and results on combinatorial game theory used in this paper,
we refer to textbooks \cite{ANW19,BCG18,Con00,HG16,Sie13}.

\subsection{Other Related Studies}
Albert and Nowakowski \cite{AN01} also characterized the $\PP$-positions in (impartial) \textsc{End-Nim}
(a.k.a.\ \textsc{Burning-Candle-at-Both-Ends} \cite{Guy95}),
in which a player removes stones from the (at most) two piles at either end of the row of $n$ piles.

Cairns and Ho \cite{CN11} showed similar results for the mis\`{e}re version of \textsc{End-Nim} and another variant called \textsc{Loop-End-Nim}, in which
the winner is the player who reduces the position to a single non-empty pile.
They also gave the Grundy numbers of positions for several cases with a few piles.

For other variants, a multi-player variant is studied by Liu and Wu \cite{LW19};
a \textsc{Wythoff} variant, in which a move is to remove stones from either end or to remove the same number of stones simultaneously from both ends, is analyzed by Fraenkel and Reisner \cite{FR09}.

\section{The Ruleset}
\label{sec:ruleset}

Let $\aint, \uint, \pint$ denote the sets of all integers, non-negative integers, and positive integers, respectively.
For $n \in \uint$, let $[n] = \{i \in \aint : 1 \leq i \leq n\}$.

Below, we introduce the formal notation and definition of \textsc{Partizan-Serial-Nim}.
We fix the number $n \in \pint$ of piles, and a permutation $\bm{\sigma} = (\sigma(1), \sigma(2), \ldots, \sigma(n))$ of $(1, 2, \ldots, n)$ that indicates Right's priority.

We identify a position in \textsc{Partizan-Serial-Nim} with the sequence indicating the numbers of stones in the $n$ piles in the position (i.e., the sequence of $n$ integers in which the $i$-th element is the number of stones in the $i$-th pile).
We define the set of all positions as $\wset \coloneqq \{(w_1, w_2, \ldots, w_n) : \forall i \in [n], w_i \in \uint\}$,
and define the set of all non-terminal positions as $\wset_+ \coloneqq \wset \setminus \{\bm{0}\}$,
where $\bm{0} \coloneqq (0, 0, \ldots, 0) \in \wset$.
Let $\nu(\bm{w})$ denote the number of non-empty piles of a position $\bm{w} \in \wset$, that is, $\nu(\bm{w}) = |\{i \in [n] : w_i > 0\}|$.

For a non-terminal position $\bm{w} \in \wset_+$, let $l(\bm{w})$ and $r(\bm{w})$ denote
the indices of Left's target pile and Right's target pile, respectively.
Namely, $l(\bm{w}) \coloneqq \min\{i \in [n] : w_i > 0\}$ and 
$r(\bm{w}) \coloneqq \operatorname*{argmin} \{\sigma(i) : i \in [n], w_i > 0\}$.

The non-terminal positions are classified into two categories by whether both players' target piles are the same or not.
Define $\wset_{=}$ as the set of all positions in which both players share their target pile, and define $\wset_{\neq}$ as the set of the others.
Namely, $\wset_{=} \coloneqq \{\bm{w} \in \wset_+ : l(\bm{w}) = r(\bm{w})\}$ and 
$\wset_{\neq} \coloneqq \{\bm{w} \in \wset_+ : l(\bm{w}) \neq r(\bm{w})\}$,
so that $\wset_+ = \wset_= \uplus \wset_{\neq}$.

Further, for $\bm{w} \in \wset_+$, let $\alpha(\bm{w})$ and $\beta(\bm{w})$ denote
the number of stones in Left's target pile and Right's target pile, respectively.
Namely, $\alpha(\bm{w}) \coloneqq w_{l(\bm{w})}$ and $\beta(\bm{w}) \coloneqq w_{r(\bm{w})}$.

For $\bm{w} \in \wset_+$ and $\gamma \in \uint$, 
we define $\ppos{\bm{w}}{\gamma}{\cdot}$ (resp.~$\ppos{\bm{w}}{\cdot}{\gamma}$) as
the position obtained by replacing the number of stones in Left's (resp.~Right's) target pile with $\gamma$.
Namely, $\ppos{\bm{w}}{\gamma}{\cdot}$ (resp.~$\ppos{\bm{w}}{\cdot}{\gamma}$) is the sequence obtained by replacing the $l(\bm{w})$-th (resp.~$r(\bm{w})$-th) element of $\bm{w}$ with $\gamma$.

Using the notation above, we now restate the definition of our ruleset formally as follows.
\begin{definition}
We define a ruleset, \textsc{Partizan-Serial-Nim}, as follows.
\begin{itemize}
\item The set of all positions is $\wset$.
\item The zero position $\bm{0} \in \wset$ is the only terminal position.
\item For a non-terminal position $\bm{w} \in \wset_+$, its Left options are $\ppos{\bm{w}}{\alpha'}{\cdot}$ for $\alpha' \in \{0, 1, 2, \ldots, \allowbreak \alpha(\bm{w})-1\}$, and its Right options are $\ppos{\bm{w}}{\cdot}{\beta'}$ for $\beta' \in \{0, 1, 2, \ldots, \beta(\bm{w})-1\}$.
In the conventional notation in combinatorial game theory,
\begin{align}
\bm{w} = \left\{ \ppos{\bm{w}}{\alpha'}{\cdot}_{0 \leq \alpha' < \alpha(\bm{w})} \mid \ppos{\bm{w}}{\cdot}{\beta'}_{0 \leq \beta' < \beta(\bm{w})} \right\}.
\end{align}
\end{itemize}
\end{definition}

Let the symbol ``$=$'' (resp.~``$\neq$'') mean that two positions are the same (resp.~not the same) as elements of $\wset$.
Note that $\bm{w}, \bm{w}' \in \wset$ with $\bm{w} \neq \bm{w}'$ in our notation may share the same game tree (e.g., $n = 2, \bm{\sigma} = (1, 2), \bm{w} = (1, 0), \bm{w}' = (0, 1)$).

\begin{example}
Figure \ref{fig:general-thre} illustrates a position $\bm{w} = (2, 3, 1, 5, 2, 3, 1, 3, 2, 3)$ of \textsc{Partizan-Serial-Nim}, where $\bm{\sigma} = (10, 1, 3, 9, 2, 5, 7, 6, 8, 4)$.
Starting from the $\bullet$ in the lower-left corner,
let the horizontal coordinates be $1, 2, \ldots, n$ from left to right,
and the vertical coordinates be $1, 2, \ldots, n$ from bottom to top.
A point with horizontal coordinate $x$ and vertical coordinate $y$ is denoted by $(x, y)$.
In particular, the coordinate of the $\bullet$ in the lower-left corner is $(1, 1)$.
For each $i \in [n]$, a circle $\bigcirc$ corresponding to the $i$-th pile is drawn at the coordinate $(i, \sigma(i))$, in which the number $w_i$ of stones in the pile is indicated.
Other components in the figure are referred to later.
Based on the figure, it can be intuitively seen that a Left move (resp.~Right move) is to remove stones from the leftmost (resp.~bottommost) non-empty pile $\bigcirc$ in the figure.
Figure \ref{fig:serial-end} illustrates particular cases of \textsc{Serial-Nim} and \textsc{Partizan-End-Nim},
where
the coordinates are based on the $\bullet$ or $\bigcirc$ in the lower-left corner as the origin $(1, 1)$, as in Figure \ref{fig:general-thre}.
More precisely, it shows a position $\bm{w} = (3, 5, 2, 3, 3, 1, 9)$, where $\bm{\sigma} = (1, 2, 3, 4, 5, 6, 7)$ in the left diagram and $\bm{\sigma} = (7, 6, 5, 4, 3, 2, 1)$ in the right diagram, which correspond to \textsc{Serial-Nim} and \textsc{Partizan-End-Nim}, respectively.
Note that all positions $\bm{w}'$ with $\nu(\bm{w}') \geq 2$ of \textsc{Serial-Nim} (resp.~\textsc{Partizan-End-Nim}) satisfy $\bm{w}' \in \wset_=$ (resp.~$\bm{w}' \in \wset_{\neq}$).
\end{example}

\begin{figure}[H]
\centering
\includegraphics[keepaspectratio,scale=0.8]{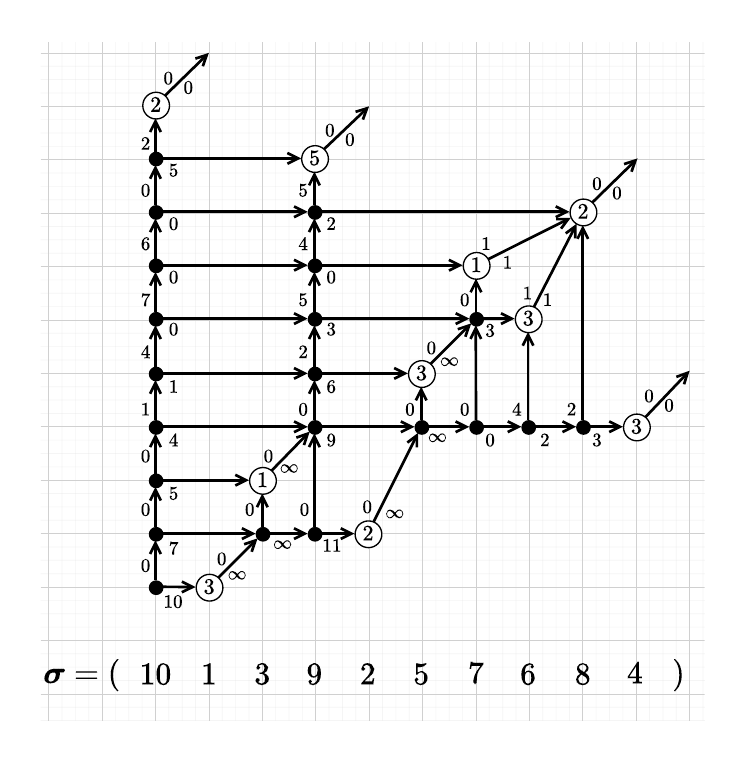}
\caption{A position $\bm{w} = (2, 3, 1, 5, 2, 3, 1, 3, 2, 3)$ of \textsc{Partizan-Serial-Nim}, where $\bm{\sigma} = (10, 1, 3, \allowbreak 9, 2, 5, 7, 6, 8, 4)$}
\label{fig:general-thre}
\end{figure}

\begin{figure}[H]
\centering
\includegraphics[keepaspectratio,scale=0.64]{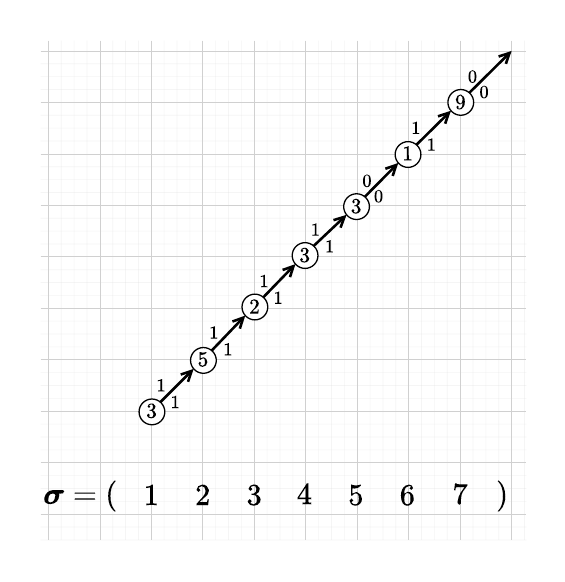}
\includegraphics[keepaspectratio,scale=0.64]{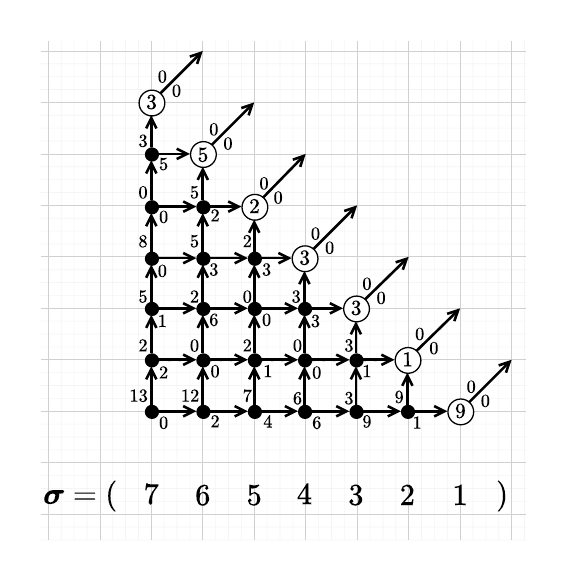}
\caption{A position $\bm{w} = (3, 5, 2, 3, 3, 1, 9)$, where  $\bm{\sigma} = (1, 2, 3, 4, 5, 6, 7)$ and 
$\bm{\sigma} = (7, 6, 5, 4, 3, 2, 1)$, which are specific examples corresponding to \textsc{Serial-Nim} and \textsc{Partizan-End-Nim}, respectively}
\label{fig:serial-end}
\end{figure}

Every position is classified into exactly one of $\LL$-, $\RR$-, $\PP$-, and $\NN$-positions according to
which player has a winning strategy when Left plays first and when Right plays first, respectively.
We write the set of all positions $\bm{w} \in \wset$ that are $\LL$-positions (resp.~$\RR$-positions, $\PP$-positions, $\NN$-positions) under normal play simply as $\LL, \RR, \PP, \NN$.
Let $o(\bm{w})$ denote the outcome class to which a position $\bm{w}$ belongs.
We write the outcome of $\bm{w}$ in mis\`{e}re play as $o^-(\bm{w})$.
We also write the outcome of $\bm{w}$ in normal play as $o^{+}(\bm{w})$ for emphasis, which is synonymous with $o(\bm{w})$.
We also refer to $o^{+}(\bm{w})$ (equivalently, $o(\bm{w})$) as the \emph{normal outcome}
and refer to $o^{-}(\bm{w})$ as the \emph{mis\`{e}re outcome}.

For convenience, we introduce the following additional notation.

First, we define $\bm{w}_L$ (resp.~$\bm{w}_R$) as the position obtained by the Left move (resp.~Right move) removing all stones from her (resp.~his) target pile of $\bm{w}$.
Namely, for $\bm{w} \in \wset_+$, we define $\bm{w}_L \coloneqq \ppos{\bm{w}}{0}{\cdot}$ and $\bm{w}_R \coloneqq \ppos{\bm{w}}{\cdot}{0}$.
Note that if $\bm{w} \in \wset_=$, then $\bm{w}_L$ and $\bm{w}_R$ mean the same position, denoted by $\bm{w}_{\ast}$.
Namely, for $\bm{w} \in \wset_=$, we define $\bm{w}_{\ast} \coloneqq \bm{w}_L = \bm{w}_R$.
For the sake of simplicity, we write $\bm{w}_{LR}$ and $\bm{w}_{L\ast}$ for $(\bm{w}_L)_R$ and $(\bm{w}_L)_{\ast}$, respectively.

For $\bm{w} \in \wset_=$ and $\gamma \in \uint$, we define $\pos{\bm{w}}{\gamma} \coloneqq \ppos{\bm{w}}{\gamma}{\cdot} = \ppos{\bm{w}}{\cdot}{\gamma}$.
For $\bm{w} \in \wset_{\neq}$ and $\alpha, \beta \in \uint$, we define $\ppos{\bm{w}}{\alpha}{\beta} \in \wset$ as the sequence obtained by simultaneously replacing $w_{l(\bm{w})}$ and $w_{r(\bm{w})}$ with $\alpha$ and $\beta$, respectively.

By definition, we observe the following two lemmas.
\begin{lemma}
\label{lem:neq-lr}
For any $\bm{w} \in \wset_{\neq}$, the following statements (i) and (ii) hold.
\begin{enumerate}[(i)]
\item $\bm{w}_{LR} = \ppos{\bm{w}}{0}{0} = \bm{w}_{RL}$.
\item For any $\alpha, \beta \in \pint$, we have
$\ppos{\bm{w}}{0}{\beta} = \ppos{\bm{w}_L}{\cdot}{\beta}$ and 
$\ppos{\bm{w}}{\alpha}{0} = \ppos{\bm{w}_R}{\alpha}{\cdot}$.
\end{enumerate}
\end{lemma}

\begin{remark}
\label{rem:seq}
Depending on the structure of $\bm{\sigma}$, a position $\bm{w}$ may be decomposed into a sequential compound of two positions.
The sequential compound $G \seq H$ of games $G$ and $H$ is the game constructed by combining $G$ and $H$ in the way that $G$ is played until it reaches a terminal position, and then $H$ is played after $G$ has finished.
In general, the following assertion {\cite[Theorem 6.1]{SU93}} regarding the outcome of $G \seq H$ holds:
for any games $G, H$, we have
\begin{align}
o^+(G \seq H)
= \begin{cases}
\LL &\,\,\text{if}\,\, o^+(H) = \LL,\\
\RR &\,\,\text{if}\,\, o^+(H) = \RR,\\
o^+(G) &\,\,\text{if}\,\, o^+(H) = \PP,\\
o^-(G) &\,\,\text{if}\,\, o^+(H) = \NN.
\end{cases}
\label{eq:t1ujq4inra3j}
\end{align}
Let $p \in [n-1]$ satisfy that for any $i, j \in [n]$ with $i \leq p < j$, it holds that $\sigma(i) < \sigma(j)$.
Then the piles indexed $p+1$ through $n$ are not played until the piles indexed $1$ through $p$ have been exhausted.
Therefore, $\bm{w}$ is represented as the sequential compound $\bm{w} = \bm{w}' \seq \bm{w}''$
of $\bm{w}' \coloneqq (w_1, \ldots, w_p, 0, \ldots, 0)$ and $\bm{w}'' = (0, \ldots, 0, w_{p+1}, \ldots, w_n)$.
In particular, if $\bm{w} \in \wset_=$, then
$\bm{w}$ is expressed as $\bm{w} = \bm{w}' \seq \bm{w}_{\ast}$,
where $\bm{w}' \coloneqq (0, 0, \ldots, 0, \alpha(\bm{w}), 0, 0, \ldots, 0)$, in which only the $l(\bm{w})$-th (equivalently, $r(\bm{w})$-th) pile is non-empty.

To obtain the mis\`{e}re outcome of $\bm{w}$, by \eqref{eq:t1ujq4inra3j}, it suffices to obtain the normal outcome 
of $\bm{w}^- \coloneqq (w_1, w_2, \ldots, w_n, 0) \seq (0, 0, \ldots, 0, m) = (w_1, w_2, \ldots, w_n, m)$ with $n+1$ piles for an arbitrary $m \in \pint$ under Right's priority permutation $\bm{\sigma}' \coloneqq (\sigma(1), \sigma(2), \ldots, \sigma(n), n+1)$.
Therefore, computing the mis\`{e}re outcome can be easily reduced to computing the normal outcome.
\end{remark}

\section{The General Case}
\label{sec:general}

In contrast to the uniform case discussed in Section \ref{sec:uniform}, this section describes the results for the general case.

\subsection{The Results for the General Case}
We provide an algorithm to compute the outcome of a given position in \textsc{Partizan-Serial-Nim} in $O(n^2)$ time
by extending the algorithm for \textsc{Partizan-End-Nim} in \cite{AN01, DKW09}.

We first describe the general idea behind the algorithm for \textsc{Partizan-End-Nim}, where $\bm{\sigma} = (n, n-1, \ldots, 1)$, as follows.
Generally speaking, the more stones Left has in her target pile (i.e., the first pile) of a position $\bm{w} = (w_1, w_2, \ldots, w_n)$,
the more options she has and thus the more advantageous $\bm{w}$ is for her.
For a fixed $(w_2, w_3, \ldots, w_n)$, there is a threshold $L(\bm{w})$ such that
Left can win playing first if and only if $w_1 > L(\bm{w})$.
In fact, $L(\bm{w}) \leq \sum_{i = 2}^n w_i$ holds because if $w_1 \geq \sum_{i = 2}^n w_i+1$, then Left can win by removing one stone from the first pile in each turn until the other piles are empty, and then removing the whole first pile when it is the only remaining non-empty pile.
The symmetric argument guarantees the existence of a threshold $R(\bm{w})$ for Right.
Then recursive formulas to derive the thresholds $L(\bm{w}')$ and $R(\bm{w}')$ of a given position $\bm{w}'$ from the ones of the smaller positions $\bm{w}'_L$ and $\bm{w}'_R$ are established. Using them, the thresholds $L(\bm{w})$ and $R(\bm{w})$ are recursively computed in $O(n^2)$ time.
The desired outcome $o(\bm{w})$ is determined by $L(\bm{w})$ and $R(\bm{w})$.

Although our results are primarily based on the idea described above, the situation is more complex.
In \textsc{Partizan-Serial-Nim}, 
even if the number of stones in the target pile is greater than the total number of stones in all other piles,
the target pile may not necessarily be the only pile remaining at the end;
therefore, the argument in the previous paragraph does not guarantee the existence of thresholds.
In fact, there exists a position in which it is impossible to win no matter how many stones the target pile has (i.e., the value of the threshold is $\infty$).
Moreover, while all positions with at least two non-empty piles in \textsc{Partizan-End-Nim} belong to $\wset_{\neq}$, we need to deal with positions in $\wset_=$ in the case of \textsc{Partizan-Serial-Nim}.

We first define counterparts of the thresholds in \cite[Definition 8]{DKW09} as follows.
\begin{definition}
\label{def:thre}
For $\bm{w} \in \wset_+$, we introduce the following definitions (i) and (ii).
\begin{enumerate}[(i)]
\item $L(\bm{w}) \coloneqq \min \{\gamma \in \uint : \ppos{\bm{w}}{\gamma}{\cdot} \in \LL \cup \PP\}$,
where $L(\bm{w}) \coloneqq \infty$ if $\ppos{\bm{w}}{\gamma}{\cdot} \not\in \LL \cup \PP$ for all $\gamma \in \uint$.
\item $R(\bm{w}) \coloneqq
\min \{\gamma \in \uint : \ppos{\bm{w}}{\cdot}{\gamma} \in \RR \cup \PP\}$,
where $R(\bm{w}) \coloneqq \infty$ if $\ppos{\bm{w}}{\cdot}{\gamma} \not\in \RR \cup \PP$ for all $\gamma \in \uint$.
\end{enumerate}
\end{definition}

Note that we are adopting definitions of $L(\cdot)$ and $R(\cdot)$ that are slightly modified versions of the ones given in \cite{DKW09} for the sake of generalization from \textsc{Partizan-End-Nim} to \textsc{Partizan-Serial-Nim}.
More specifically, $L(\bm{w}) = \tilde{L}(\bm{w}_L)$ and $R(\bm{w}) =  \tilde{R}(\bm{w}_R)$ hold for $\bm{w} \in \wset_{\neq}$, where $\tilde{L}(\cdot)$ and $\tilde{R}(\cdot)$ mean $L(\cdot)$ and $R(\cdot)$ in the sense of \cite{DKW09}.

The main theorem is as follows. This is a generalization of \cite[Observation 9 and Proposition 14]{DKW09}.

\begin{theorem}
\label{thm:outcome}
For any $\bm{w} \in \wset_+$, the following statements (i)--(iii) hold.
\begin{enumerate}[(i)]
\item If $\bm{w} \in \wset_=$, then
\begin{align}
L(\bm{w}) =
\begin{cases}
0 &\,\,\text{if}\,\, o(\bm{w}_{\ast}) = \LL,\\
\infty &\,\,\text{if}\,\, o(\bm{w}_{\ast}) = \RR,\\
0 &\,\,\text{if}\,\, o(\bm{w}_{\ast}) = \PP,\\
1 &\,\,\text{if}\,\, o(\bm{w}_{\ast}) = \NN,
\end{cases} \qquad
R(\bm{w}) =
\begin{cases}
\infty &\,\,\text{if}\,\, o(\bm{w}_{\ast}) = \LL,\\
0 &\,\,\text{if}\,\, o(\bm{w}_{\ast}) = \RR,\\
0 &\,\,\text{if}\,\, o(\bm{w}_{\ast}) = \PP,\\
1 &\,\,\text{if}\,\, o(\bm{w}_{\ast}) = \NN.
\end{cases}
\label{eq:1dt0k5g1gy5l}
\end{align}

\item If $\bm{w} \in \wset_{\neq}$, then
\begin{align}
L(\bm{w}) &=
\begin{cases}
0 &\,\,\text{if}\,\, \beta(\bm{w}) \leq R(\bm{w}_L),\\
L(\bm{w}_R) + (\beta(\bm{w})-R(\bm{w}_L)) &\,\,\text{if}\,\, \beta(\bm{w}) > R(\bm{w}_L),
\end{cases} \label{eq:l7amcfizeu8d}\\
R(\bm{w}) &=
\begin{cases}
0 &\,\,\text{if}\,\, \alpha(\bm{w}) \leq L(\bm{w}_R),\\
R(\bm{w}_L) + (\alpha(\bm{w})-L(\bm{w}_R)) &\,\,\text{if}\,\, \alpha(\bm{w}) > L(\bm{w}_R).
\end{cases}
\end{align}

\item We have
\begin{align}
o(\bm{w}) = 
\begin{cases}
\LL &\,\,\text{if}\,\, \alpha(\bm{w}) > L(\bm{w}), \beta(\bm{w}) \leq R(\bm{w}),\\
\RR &\,\,\text{if}\,\, \alpha(\bm{w}) \leq L(\bm{w}), \beta(\bm{w}) > R(\bm{w}),\\
\NN &\,\,\text{if}\,\, \alpha(\bm{w}) > L(\bm{w}), \beta(\bm{w}) > R(\bm{w}),\\
\PP &\,\,\text{if}\,\, \alpha(\bm{w}) \leq L(\bm{w}), \beta(\bm{w}) \leq R(\bm{w}).
\end{cases}
\label{eq:8nm1tpxiz7d2}
\end{align}
\end{enumerate}
\end{theorem}

Theorem \ref{thm:outcome} yields the following recursive algorithm to obtain the outcome of a given position $\bm{w}$.
For the base case $\nu(\bm{w}) = 1$, the position $\bm{w}$ is in $\wset_=$, and thus
we have $L(\bm{w}) = R(\bm{w}) = 0$ from $o(\bm{w}_{\ast}) = o(\bm{0}) = \PP$ and Theorem \ref{thm:outcome} (i); this leads to $o(\bm{w}) = \NN$ by Theorem \ref{thm:outcome} (iii).
For the case $\nu(\bm{w}) \geq 2$, we obtain $L(\bm{w})$ and $R(\bm{w})$ from the information on $\bm{w}_L$ and $\bm{w}_R$ by applying Theorem \ref{thm:outcome} (i) and (ii) corresponding to whether $\bm{w} \in \wset_=$ or $\bm{w} \in \wset_{\neq}$.
Then $o(\bm{w})$ is obtained from $L(\bm{w})$ and $R(\bm{w})$ by Theorem \ref{thm:outcome} (iii).

Let $\sset{\bm{w}}$ denote the set of all positions that can be reached by repeatedly applying any number of operations in the form $\bm{w}' \mapsto \bm{w}'_L$ or $\bm{w}' \mapsto \bm{w}'_R$ in any order.
To compute $o(\bm{w})$ by using Theorem \ref{thm:outcome}, we must compute $L(\bm{w}')$ and $R(\bm{w}')$ recursively for all $\bm{w}' \in \sset{\bm{w}}$.
By definition, a position $\bm{w}' \in \sset{\bm{w}}$ is in the form obtained by replacing some elements of $\bm{w}$ with $0$.

Let $\bm{w}' = (w'_1, w'_2, \ldots, w'_n) \in \sset{\bm{w}}$ satisfy $\left(l(\bm{w}'), \sigma(r(\bm{w}'))\right) = (x, y)$.
Then
\begin{align}
w'_1 = w'_2 =  \cdots =  w'_{x-1} = 0,\quad
w'_{\sigma^{-1}(1)} = w'_{\sigma^{-1}(2)} = \cdots = w'_{\sigma^{-1}(y-1)} = 0, \label{eq:leht74hx6qzo}
\end{align}
and the other elements are not changed from the original position $\bm{w}$ by the definition of the ruleset.
Therefore, for a given $(x, y) \in [n] \times [n]$, the position $\bm{w}' \in \sset{\bm{w}}$ such that $\left(l(\bm{w}'), \sigma(r(\bm{w}'))\right) = (x, y)$ is uniquely determined if it exists.
Then we write it as $\spos{\bm{w}}xy$.
By \eqref{eq:leht74hx6qzo}, for $\spos{\bm{w}}xy$ to exist,
it must hold that
\begin{align}
x \not \in \{\sigma^{-1}(1), \sigma^{-1}(2), \ldots, \sigma^{-1}(y-1)\},\quad
\sigma^{-1}(y) \not \in \{1, 2, \ldots, x-1\}.
\end{align}
Therefore, the condition for $\spos{\bm{w}}xy$ to exist is $\sigma(x) \geq y$ and $x \leq \sigma^{-1}(y)$,
which is equivalent to $x \leq x'$ and $\sigma(x) \geq \sigma(x')$, where $x' \coloneqq \sigma^{-1}(y)$.
Namely, $\sset{\bm{w}} = \{\spos{\bm{w}}{x}{\sigma(x')} : (x, x') \in [n] \times [n], x \leq x', \sigma(x) \geq \sigma(x')\}$, 
whose cardinality is given as $\mathrm{inv}(\bm{\sigma}) + n \leq n(n+1)/2$,
where $\mathrm{inv}(\bm{\sigma})$ denotes the inversion number of the permutation $\bm{\sigma}$.
In particular, we can compute $o(\bm{w})$ by applying the recursive formula of Theorem \ref{thm:outcome} $O(n^2)$ times.

In Figs.~\ref{fig:general-thre} and \ref{fig:serial-end}, for each $\spos{\bm{w}}x{\sigma(x')} \in \sset{\bm{w}}$ with $x < x'$,
a point $\bullet$ is drawn at the coordinate $(x, \sigma(x'))$ corresponding to $\spos{\bm{w}}x{\sigma(x')}$.
Equivalently, there is a $\bullet$ at each intersection point of the half-lines extending left and down from every $\bigcirc$.
Also, note that the point corresponding to each $\spos{\bm{w}}x{\sigma(x')} \in \sset{\bm{w}}$ with $x = x'$ is the $\bigcirc$ at the coordinate $(x, \sigma(x'))$.
Therefore, a $\bigcirc$ corresponds to both the $x$-th pile and the position $\spos{\bm{w}}x{\sigma(x)} \in \sset{\bm{w}}$.

The remaining non-empty piles in a position $\spos{\bm{w}}x{\sigma(x')} \in \sset{\bm{w}}$ are exactly the non-empty piles corresponding to the $\bigcirc$s located in the upper-right region starting from the coordinate $(x, \sigma(x'))$.

A position $\bm{w}'$ corresponding to a $\bullet$ belongs to $\wset_{\neq}$;
the destination of the right (resp.~up) arrow going from it corresponds to $\bm{w}'_L$ (resp.~$\bm{w}'_R$);
the values of $L(\bm{w}')$ and $ R(\bm{w}')$ are indicated as labels of the arrows as shown in the left figure of Fig.~\ref{fig:label}.
A position corresponding to a $\bigcirc$ belongs to $\wset_=$;
the destination of the diagonal arrow going from it corresponds to $\bm{w}'_{\ast}$, and if the diagonal arrow has no destination, then $\bm{w}'_{\ast} = \bm{0}$;
the values $L(\bm{w}')$ and $R(\bm{w}')$ are indicated as labels of the arrows as shown in the right figure of Fig.~\ref{fig:label}.

\begin{figure}[H]
\centering
\includegraphics[keepaspectratio,scale=1.5]{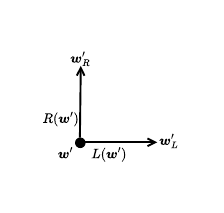}
\includegraphics[keepaspectratio,scale=1.5]{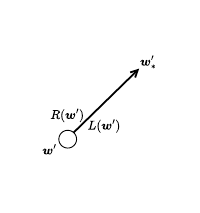}
\caption{The labels and destinations of arrows in the figures}
\label{fig:label}
\end{figure}

\begin{example}
Let $\bm{\sigma} = (10, 1, 3, 9, 2, 5, 7, 6, 8, 4)$ and $\bm{w} = (2, 3, 1, 5, 2, 3, 1, 3, 2, 3)$ as shown in Fig.~\ref{fig:general-thre}.
We suppose that it is known that
\begin{align}
(L(\bm{w}_R), R(\bm{w}_R)) = (7, 0), \quad (L(\bm{w}_{L\ast}), R(\bm{w}_{L\ast})) = (\infty, 0),
\end{align}
and we show the last few steps of the recursive algorithm to obtain $o(\bm{w})$.
Since
$\alpha(\bm{w}_{L\ast}) = w_3 = 1 \leq \infty = L(\bm{w}_{L\ast})$ and $\beta(\bm{w}_{L\ast}) = w_5 = 2 > 0 = R(\bm{w}_{L\ast})$,
Theorem \ref{thm:outcome} (iii) leads to
\begin{align}
o(\bm{w}_{L\ast}) = \RR.
\end{align}
Therefore, we obtain
\begin{align}
(L(\bm{w}_L), R(\bm{w}_L)) = (\infty, 0)
\end{align}
by applying Theorem \ref{thm:outcome} (i).
Since $\beta(\bm{w}) = w_2 =  3 > 0 = R(\bm{w}_L)$, we have
\begin{align}
L(\bm{w}) = L(\bm{w}_R) + (\beta(\bm{w}) - R(\bm{w}_L)) = 7 + (3-0) = 10
\end{align} by Theorem \ref{thm:outcome} (ii).
Since $\alpha(\bm{w}) = w_1 = 2 < 7 = L(\bm{w}_R)$, we have
\begin{align}
R(\bm{w}) = 0
\end{align}
by Theorem \ref{thm:outcome} (ii).
Therefore, we see that $\alpha(\bm{w}) = w_1 = 2 \leq 10 = L(\bm{w})$ and $\beta(\bm{w}) = w_2 = 3 > 0 = R(\bm{w})$,
so that
\begin{align}
o(\bm{w}) = \RR
\end{align}
by applying Theorem \ref{thm:outcome} (iii).
\end{example}

\subsection{Proof of Theorem \ref{thm:outcome}}

We first see that $L(\bm{w})$ and $R(\bm{w})$ in Definition \ref{def:thre} indeed provide the thresholds for the number of stones in the target pile that allows the player to win when playing first, as the following lemma.

\begin{lemma}
\label{lem:thre}
For any $\bm{w} \in \wset_+$ and $\gamma \in \pint$, the following equivalences (a) and (b) hold.
\begin{enumerate}[(a)]
\item $\ppos{\bm{w}}{\gamma}{\cdot} \in \LL \cup \NN \iff \gamma > L(\bm{w})$.
\item $\ppos{\bm{w}}{\cdot}{\gamma} \in \RR \cup \NN \iff \gamma > R(\bm{w})$.
\end{enumerate}
\end{lemma}
Note that $\gamma > L(\bm{w})$ and $\gamma > R(\bm{w})$ in the statements above
imply $L(\bm{w}) \neq \infty$ and $R(\bm{w}) \neq \infty$, respectively, since $\gamma$ is an integer.

\begin{proof}[Proof of Lemma \ref{lem:thre}]
By symmetry, it suffices to prove only the equivalence (a).

($\implies$)
Since $\ppos{\bm{w}}{\gamma}{\cdot} \in \LL \cup \NN$, the position $\ppos{\bm{w}}{\gamma}{\cdot}$ has a Left option $\ppos{\bm{w}}{\gamma'}{\cdot} \in \LL \cup \PP$ with $0 \leq \gamma' < \gamma$, so that $L(\bm{w}) \leq \gamma' < \gamma$.

($\impliedby$)
By $\gamma > L(\bm{w})$, the position $\ppos{\bm{w}}{\gamma}{\cdot}$ has the Left option $\ppos{\bm{w}}{L(\bm{w})}{\cdot}$,
which is in $\LL \cup \PP$ by Definition \ref{def:thre} (i).
\end{proof}

The next lemma describes a relation between the outcome $o(\bm{w})$ and the thresholds for $\bm{w}_L$ and $\bm{w}_R$ in the case $\bm{w} \in \wset_{\neq}$.

\begin{lemma}
\label{lem:tri-point}
For any $\bm{w} \in \wset_{\neq}$ and $\alpha, \beta \in \pint$, we have
\begin{align}
o(\ppos{\bm{w}}{\alpha}{\beta}) =
\begin{cases}
\NN &\,\,\text{if}\,\, \alpha \leq L(\bm{w}_R), \beta \leq R(\bm{w}_L),\\
\PP &\,\,\text{if}\,\, \alpha = L(\bm{w}_R) + \gamma, \beta = R(\bm{w}_L) + \gamma \,\,\text{for some}\,\, \gamma \in \pint,\\
\LL &\,\,\text{if}\,\, \alpha = L(\bm{w}_R) + \gamma, \beta < R(\bm{w}_L) + \gamma \,\,\text{for some}\,\, \gamma \in \pint,\\
\RR &\,\,\text{if}\,\, \alpha < L(\bm{w}_R) + \gamma, \beta = R(\bm{w}_L) + \gamma \,\,\text{for some}\,\, \gamma \in \pint.
\end{cases}
\end{align}
\end{lemma}

The proof of Lemma \ref{lem:tri-point} is essentially identical to the proof of \cite[Proposition 11]{DKW09}.

\begin{proof}[Proof of Lemma \ref{lem:tri-point}]
We consider the four cases separately as follows.

\begin{itemize}
\item The case $\alpha \leq L(\bm{w}_R), \beta \leq R(\bm{w}_L)$:
We have $\ppos{\bm{w}}{\alpha}{\beta} \in \LL \cup \NN$ because it has the Left option $\ppos{\bm{w}}{0}{\beta}$,
which satisfies
\begin{align}
\ppos{\bm{w}}{0}{\beta} \eqlab{A}= \ppos{\bm{w}_L}{\cdot}{\beta} \eqlab{B}\in \LL \cup \PP,
\end{align}
where
(A) follows from Lemma \ref{lem:neq-lr} (ii),
and (B) follows from $\beta \leq R(\bm{w}_L)$ and Lemma \ref{lem:thre} (b).
By symmetry, the assumption $\alpha \leq L(\bm{w}_R)$ yields $\ppos{\bm{w}}{\alpha}{\beta} \in \RR \cup \NN$, so that $\ppos{\bm{w}}{\alpha}{\beta} \in \NN$ as desired.

\item The case where $\alpha = L(\bm{w}_R) + \gamma, \beta = R(\bm{w}_L) + \gamma$ for some $\gamma \in \pint$:
We prove $\ppos{\bm{w}}{\alpha}{\beta} \in \PP$ by induction on $\gamma$.
By symmetry, it suffices to show that $\ppos{\bm{w}}{\alpha}{\beta} \in \RR \cup \PP$, that is, Left loses in $\ppos{\bm{w}}{\alpha}{\beta}$ playing first.
Let $\ppos{\bm{w}}{\alpha'}{\beta}$ be an arbitrary Left option of $\ppos{\bm{w}}{\alpha}{\beta}$, where $0 \leq \alpha' < \alpha$.
We consider the following three cases.
\begin{itemize}
\item The case $\alpha' = 0$: We have
\begin{align}
\ppos{\bm{w}}{\alpha'}{\beta}
= \ppos{\bm{w}}{0}{\beta}
\eqlab{A}= \ppos{\bm{w}_L}{\cdot}{\beta}
\eqlab{B}\in \RR\cup\NN,
\end{align}
where
(A) follows from Lemma \ref{lem:neq-lr} (ii),
and (B) follows from $\beta = R(\bm{w}_L) + \gamma > R(\bm{w}_L)$ and Lemma \ref{lem:thre} (b).
\item The case $0 < \alpha' \leq L(\bm{w}_R)$: The position $\ppos{\bm{w}}{\alpha'}{\beta}$ is in $\RR \cup \NN$ because it has the Right option $\ppos{\bm{w}}{\alpha'}{0}$,
which satisfies
\begin{align}
\ppos{\bm{w}}{\alpha'}{0} \eqlab{A}= \ppos{\bm{w}_R}{\alpha'}{\cdot} \eqlab{B}\in \RR \cup \PP,
\end{align}
where (A) follows from Lemma \ref{lem:neq-lr} (ii),
and (B) follows from $\alpha' \leq L(\bm{w}_R)$ and Lemma \ref{lem:thre} (a).

\item The case where $\alpha' = L(\bm{w}_R) + \gamma'$ for some $\gamma'$ with $0 < \gamma' < \gamma$:
The position $\ppos{\bm{w}}{\alpha'}{\beta} = \ppos{\bm{w}}{L(\bm{w}_R) + \gamma'}{R(\bm{w}_L) + \gamma}$ is in $\RR \cup \NN$ because it has the Right option $\ppos{\bm{w}}{L(\bm{w}_R) + \gamma'}{R(\bm{w}_L) + \gamma'}$,
which is in $\PP$ by the induction hypothesis on $\gamma$.
\end{itemize}

\item The case where $\alpha = L(\bm{w}_R) + \gamma, \beta < R(\bm{w}_L) + \gamma$ for some $\gamma \in \pint$:
To show $\ppos{\bm{w}}{\alpha}{\beta} \in \LL$, we show that Left wins in $\ppos{\bm{w}}{\alpha}{\beta}$ for both cases where Left plays first and the case where Left plays second, as follows.
\begin{itemize}
\item The case where Left plays first:
If $\beta \leq R(\bm{w}_L)$, then Left has the move to $\ppos{\bm{w}}{0}{\beta}$,
which satisfies
\begin{align}
\ppos{\bm{w}}{0}{\beta}
\eqlab{A}= \ppos{\bm{w}_L}{\cdot}{\beta}
\eqlab{B}\in \LL \cup \PP,
\end{align}
where
(A) follows from Lemma \ref{lem:neq-lr} (ii),
and (B) follows from $\beta \leq R(\bm{w}_L)$ and Lemma \ref{lem:thre} (b).
If $\beta = R(\bm{w}_L) + \gamma'$ with $0 < \gamma' < \gamma$,
then Left has the move from $\ppos{\bm{w}}{\alpha}{\beta} = \ppos{\bm{w}}{L(\bm{w}_{R})+\gamma}{R(\bm{w}_L) + \gamma'}$
to $\ppos{\bm{w}}{L(\bm{w}_{R})+\gamma'}{R(\bm{w}_L) + \gamma'}$, which is in $\PP$ as shown above.
\item The case where Left plays second:
Let $\ppos{\bm{w}}{\alpha}{\beta'}$ be an arbitrary Right option of $\ppos{\bm{w}}{\alpha}{\beta}$, where $\beta' < \beta$.
If $\beta' = 0$, then
\begin{align}
\ppos{\bm{w}}{\alpha}{\beta'}
= \ppos{\bm{w}}{\alpha}{0}
\eqlab{A}= \ppos{\bm{w}_R}{\alpha}{ \cdot}
\eqlab{B}\in \LL \cup \NN,
\end{align}
where
(A) follows from Lemma \ref{lem:neq-lr} (ii),
and (B) follows from $\alpha = L(\bm{w}_R) + \gamma > L(\bm{w}_R)$ and Lemma \ref{lem:thre} (a).
If $\beta' > 0$, then $\beta' < \beta < R(\bm{w}_L) + \gamma$ and thus it is reduced to the case where Left plays first shown above,
so that $\ppos{\bm{w}}{\alpha}{\beta'} \in \LL \cup \NN$ as desired.
\end{itemize}

\item The case where $\alpha < L(\bm{w}_R) + \gamma, \beta = R(\bm{w}_L) + \gamma$ for some $\gamma \in \pint$:
The assertion is shown by a symmetric argument to the above case $\alpha = L(\bm{w}_R) + \gamma, \beta < R(\bm{w}_L) + \gamma$.
\end{itemize}
\end{proof}

Using the above lemmas, we prove the desired theorem as follows.

\begin{proof}[Proof of Theorem \ref{thm:outcome}]
(Proof of (i)) 
By symmetry, we only show the first equation.
As stated in Remark \ref{rem:seq},
the position $\bm{w} \in \wset_=$ is represented as $\bm{w} = \bm{w}' \seq \bm{w}_{\ast}$ by
the position $\bm{w}' \coloneqq (0, 0, \ldots, 0, \alpha(\bm{w}), 0, 0, \ldots, 0)$ with only the $l(\bm{w})$-th (equivalently, $r(\bm{w})$-th) pile being non-empty and the remaining part $\bm{w}_{\ast}$.
The outcome of $\pos{\bm{w}'}{\gamma}$ is given as
\begin{align}
o^+(\pos{\bm{w}'}{\gamma}) = 
\begin{cases}
\PP &\,\,\text{if}\,\,\gamma = 0,\\
\NN &\,\,\text{if}\,\,\gamma \geq 1,
\end{cases}
\qquad
o^-(\pos{\bm{w}'}{\gamma}) = 
\begin{cases}
\NN &\,\,\text{if}\,\,\gamma = 0,\\
\PP &\,\,\text{if}\,\,\gamma = 1,\\
\NN &\,\,\text{if}\,\,\gamma \geq 2.
\end{cases}
\label{eq:ep9r308lb9f0}
\end{align}
Combining \eqref{eq:ep9r308lb9f0} and \eqref{eq:t1ujq4inra3j} yields the desired result.

(Proof of (ii)) By symmetry, we only show the first equation \eqref{eq:l7amcfizeu8d}.
We have
\begin{align}
\beta(\bm{w}) > R(\bm{w}_L)
\eqlab{A}\iff \ppos{\bm{w}_L}{\cdot}{\beta(\bm{w})} \in \RR \cup \NN
\eqlab{B}\iff \ppos{\bm{w}}{0}{\cdot} \in \RR \cup \NN,
\end{align}
where
(A) follows from Lemma \ref{lem:thre} (b),
and (B) follows from Lemma \ref{lem:neq-lr} (ii) and the assumption $\bm{w} \in \wset_{\neq}$.
Equivalently,
\begin{align}
\ppos{\bm{w}}{0}{\cdot} \in \LL \cup \PP \iff \beta(\bm{w})\leq R(\bm{w}_L).
\end{align}
Hence, by Definition \ref{def:thre} (i),
it holds that $L(\bm{w}) = 0$ if and only if $\beta(\bm{w}) \leq R(\bm{w}_L)$.
This shows the first case of \eqref{eq:l7amcfizeu8d}.
To show the other case, we suppose $\beta(\bm{w}) > R(\bm{w}_L)$. Then
for any $\alpha \in \pint$, we have
\begin{align}
\ppos{\bm{w}}{\alpha}{\cdot} \in \LL \cup \PP
\eqlab{A}\iff \alpha \geq L(\bm{w}_R) + (\beta(\bm{w}) - R(\bm{w}_L)),
\end{align}
where
(A) follows from Lemma \ref{lem:tri-point} with $\gamma \coloneqq \beta(\bm{w}) - R(\bm{w}_L) > 0$.
This shows $L(\bm{w}) = L(\bm{w}_R) + (\beta(\bm{w}) - R(\bm{w}_L))$ from Definition \ref{def:thre} (i).

(Proof of (iii)) Directly from Lemma \ref{lem:thre}.
\end{proof}

\section{The Uniform Case}
\label{sec:uniform}

In this section, we consider the particular case where
all non-empty piles have uniformly $m$ stones for a fixed constant $m$.
We define $\wset^{(m)}$ as the set of all such positions.
More formally, for $m \in \pint$, we define
$\wset^{(m)} \coloneqq \{(w_1, w_2, \ldots, w_n) \in \wset : \forall i \in [n], w_i \in \{0, m\}\}$.
Similarly, 
$\wset^{(m)}_+ \coloneqq \wset^{(m)} \cap \wset_+$,
$\wset^{(m)}_= \coloneqq \wset^{(m)} \cap \wset_=$,
$\wset^{(m)}_{\neq} \coloneqq \wset^{(m)} \cap \wset_{\neq}$.
Note that for any $\bm{w} \in \wset^{(m)}$, we have $\sset{\bm{w}} \subseteq \wset^{(m)}$.

\subsection{The Results for the Uniform Case}
The case $m = 1$ is easy since then each move removes a whole pile consisting of one stone,
and the outcome is determined only by the parity of the number of non-empty piles as follows.

\begin{proposition}
For any $\bm{w} \in \wset^{(1)}$, we have
\begin{align}
o^+(\bm{w}) =
\begin{cases}
\PP &\,\,\text{if}\,\,\nu(\bm{w}) \equiv 0 \pmod{2},\\
\NN &\,\,\text{if}\,\,\nu(\bm{w}) \equiv 1 \pmod{2},
\end{cases}\\
o^-(\bm{w}) =
\begin{cases}
\NN &\,\,\text{if}\,\,\nu(\bm{w}) \equiv 0 \pmod{2},\\
\PP &\,\,\text{if}\,\,\nu(\bm{w}) \equiv 1 \pmod{2}.
\end{cases}
\end{align}
\end{proposition}

Below, we consider the case $m \geq 2$.

\begin{figure}[H]
\centering
\includegraphics[keepaspectratio,scale=0.8]{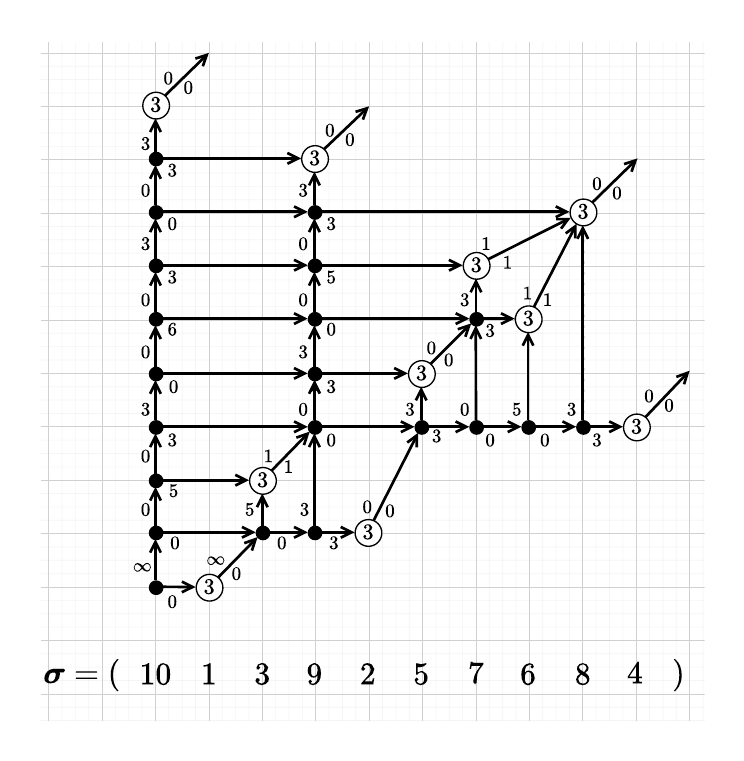}
\caption{The values of $L(\cdot)$ and $R(\cdot)$ for the uniform case where $\bm{\sigma} = (10, 1, 3, 9, 2, 5, 7, 6, 8, 4)$ and $m = 3$}
\label{fig:uniform-thre}
\end{figure}

Figure \ref{fig:uniform-thre} illustrates the uniform case where $\bm{\sigma} = (10, 1, 3, 9, 2, 5, 7, 6, 8, 4)$, as in Fig.~\ref{fig:general-thre}, and $m = 3$.

The following is the main theorem to determine the outcome of a given position in $O(n)$ time.

\begin{theorem}
\label{thm:const-outcome}
For any integer $m \geq 2$ and $\bm{w} \in \wset^{(m)}_+$, we have
\begin{align}
o(\bm{w}) =
\begin{cases}
\phi_{\PP\to\NN}(o(\bm{w}_{\ast})) &\,\,\text{if}\,\, \bm{w} \in \wset_=^{(m)},\\
\phi_{\NN\to\PP}(o(\bm{w}_{LR})) &\,\,\text{if}\,\, \bm{w} \in \wset_{\neq}^{(m)}, \bm{w}_L \in \wset_=^{(m)}, \bm{w}_R \in \wset_=^{(m)},\\
\phi_{\NN\to\LL}(o(\bm{w}_{LR})) &\,\,\text{if}\,\, \bm{w} \in \wset_{\neq}^{(m)}, \bm{w}_L \in \wset_{\neq}^{(m)}, \bm{w}_R \in \wset_=^{(m)},\\
\phi_{\NN\to\RR}(o(\bm{w}_{LR})) &\,\,\text{if}\,\, \bm{w} \in \wset_{\neq}^{(m)}, \bm{w}_L \in \wset_=^{(m)}, \bm{w}_R \in \wset_{\neq}^{(m)},\\
o(\bm{w}_{LR}) &\,\,\text{if}\,\, \bm{w} \in \wset_{\neq}^{(m)}, \bm{w}_L \in \wset_{\neq}^{(m)}, \bm{w}_R \in \wset_{\neq}^{(m)},
\end{cases}
\end{align}
where $\phi_{\PP \to \NN} \colon \{\LL, \RR, \NN, \PP\} \to \{\LL, \RR, \NN, \PP\}$ is
the mapping which maps $\PP$ to $\NN$ and leaves the others unchanged, that is,
$\LL \mapsto \LL, \RR \mapsto \RR, \NN \mapsto \NN, \PP \mapsto \NN$.
The other mappings $\phi_{\NN\to\PP}, \phi_{\NN\to\LL}, \phi_{\NN\to\RR}$ are defined in the same way.
\end{theorem}

By Theorem \ref{thm:const-outcome}, the outcome does not depend on $m$ as long as $m \geq 2$.
Namely, we have the following corollary.

\begin{corollary}
Let an integer $m \geq 2$ and $\bm{w}^{(m)} \in \wset^{(m)}_+$ be arbitrary.
Define $\bm{w}^{(2)} \in \wset^{(2)}$ as the position obtained by replacing every non-zero element $m$ in $\bm{w}^{(m)}$ with $2$.
Then we have $o(\bm{w}^{(m)}) = o(\bm{w}^{(2)})$.
\end{corollary}

Moreover, the mis\`{e}re outcome of a non-terminal position is equal to the normal outcome as follows.

\begin{corollary}
\label{cor:uniform-misere}
For any integer $m \geq 2$, the position $\bm{w} \in \wset^{(m)}_+$ satisfies $o^+(\bm{w}) = o^-(\bm{w})$.
\end{corollary}

\begin{proof}[Proof of Corollary \ref{cor:uniform-misere}]
Let $\bm{\sigma} = (\sigma(1), \sigma(2), \ldots, \sigma(n), n+1)$ and $\bm{e} = (0, 0, \ldots, 0, m)$ of length $n+1$.
Then for any $\bm{w}' = (w'_1, w'_2, \ldots, w'_n, 0) \in \wset^{(m)}$, we have
\begin{align}
o^-(\bm{w}') = o^+(\bm{w}' \seq \bm{e}) = o^+(w'_1, w'_2, \ldots, w'_n, m) \label{eq:2a4a4ldnvjjb}
\end{align}
as stated in Remark \ref{rem:seq}.

We prove the assertion by induction on $\nu(\bm{w})$.
If $\bm{w} \in \wset_=^{(m)}$ and $\bm{w}_{\ast} \in \wset_+^{(m)}$, then
\begin{align}
o^-(\bm{w})
&\eqlab{A}= o^+(\bm{w} \seq \bm{e})\\
&\eqlab{B}= \phi_{\PP\to\NN}(o^+( (\bm{w} \seq \bm{e})_{\ast}) )\\
&= \phi_{\PP\to\NN}(o^+( \bm{w}_{\ast} \seq \bm{e}) )\\
&\eqlab{C}= \phi_{\PP\to\NN}(o^-( \bm{w}_{\ast}) )\\
&\eqlab{D}= \phi_{\PP\to\NN}(o^+( \bm{w}_{\ast}) )\\
&\eqlab{E}= o^+(\bm{w}),
\end{align}
where 
(A) follows from \eqref{eq:2a4a4ldnvjjb},
(B) follows from the first case of Theorem \ref{thm:const-outcome} since $\bm{w} \seq \bm{e} \in \wset_=^{(m)}$ holds by $\bm{w} \in \wset_=^{(m)}$,
(C) follows from \eqref{eq:2a4a4ldnvjjb},
(D) follows from the induction hypothesis,
and (E) follows from the first case of Theorem \ref{thm:const-outcome} since $\bm{w} \in \wset_=^{(m)}$.

The case $\bm{w} \in \wset_{\neq}^{(m)}, \bm{w}_{LR} \in \wset_+^{(m)}$ is shown similarly.
The remaining cases to consider are as follows.
\begin{itemize}
\item The case $\bm{w} \in \wset_=^{(m)}, \bm{w}_{\ast} = \bm{0}$: We see $o^-(\bm{w}) = o^+(\bm{w})$ as follows:
\begin{align}
o^-(\bm{w})
&\eqlab{A}= o^+(\bm{w} \seq \bm{e})\\
&\eqlab{B}= \phi_{\PP \to \NN}(o^+( (\bm{w} \seq \bm{e})_{\ast}) )\\
&= \phi_{\PP \to \NN}(o^+(\bm{w}_{\ast} \seq \bm{e}) )\\
&= \phi_{\PP \to \NN}(o^+(\bm{0} \seq \bm{e}) )\\
&= \phi_{\PP \to \NN}(o^+(\bm{e}) )\\
&= \phi_{\PP \to \NN}(\NN)\\
&= \NN,\\
o^+(\bm{w})
&\eqlab{C}= \phi_{\PP \to \NN}(o^+(\bm{w}_{\ast}))
= \phi_{\PP \to \NN}(o^+(\bm{0}))
= \phi_{\PP \to \NN}(\PP)
= \NN,
\end{align}
where
(A) follows from \eqref{eq:2a4a4ldnvjjb},
(B) follows from the first case of Theorem \ref{thm:const-outcome} since $\bm{w} \seq \bm{e} \in \wset_=^{(m)}$ holds by $\bm{w} \in \wset_=^{(m)}$,
and (C) follows from the first case of Theorem \ref{thm:const-outcome} since $\bm{w} \in \wset_=^{(m)}$.

\item The case $\bm{w} \in \wset_{\neq}^{(m)}, \bm{w}_{LR} = \bm{0}$:
Then $\nu(\bm{w}) = 2$, so that $\bm{w}_L, \bm{w}_R \in \wset_=^{(m)}$;
therefore, $\bm{w}_L \seq \bm{e}, \bm{w}_R \seq \bm{e} \in \wset_=^{(m)}$.
We see $o^-(\bm{w}) = o^+(\bm{w})$ as follows:
\begin{align}
o^-(\bm{w})
&\eqlab{A}= o^+(\bm{w} \seq \bm{e})\\
&\eqlab{B}= \phi_{\NN \to \PP}(o^+((\bm{w} \seq \bm{e})_{LR}) ) \\
&= \phi_{\NN \to \PP}(o^+(\bm{w}_{LR} \seq \bm{e}) ) \\
&= \phi_{\NN \to \PP}(o^+(\bm{0} \seq \bm{e}) ) \\
&= \phi_{\NN \to \PP}(o^+(\bm{e}) )\\
&= \phi_{\NN \to \PP}(\NN)\\
&= \PP,\\
o^+(\bm{w})
&\eqlab{C}= \phi_{\NN \to \PP}(o^+(\bm{w}_{LR}))
= \phi_{\NN \to \PP}(o^+(\bm{0}))
= \phi_{\NN \to \PP}(\PP)
= \PP,
\end{align}
where
(A) follows from \eqref{eq:2a4a4ldnvjjb},
(B) follows from the second case of Theorem \ref{thm:const-outcome} since $\bm{w}_L \seq \bm{e}, \bm{w}_R \seq \bm{e} \in \wset_=^{(m)}$,
and (C) follows from the second case of Theorem \ref{thm:const-outcome} since $\bm{w}_L, \bm{w}_R \in \wset_=^{(m)}$.
\end{itemize}
\end{proof}

\begin{figure}[H]
\centering
\includegraphics[keepaspectratio,scale=0.8]{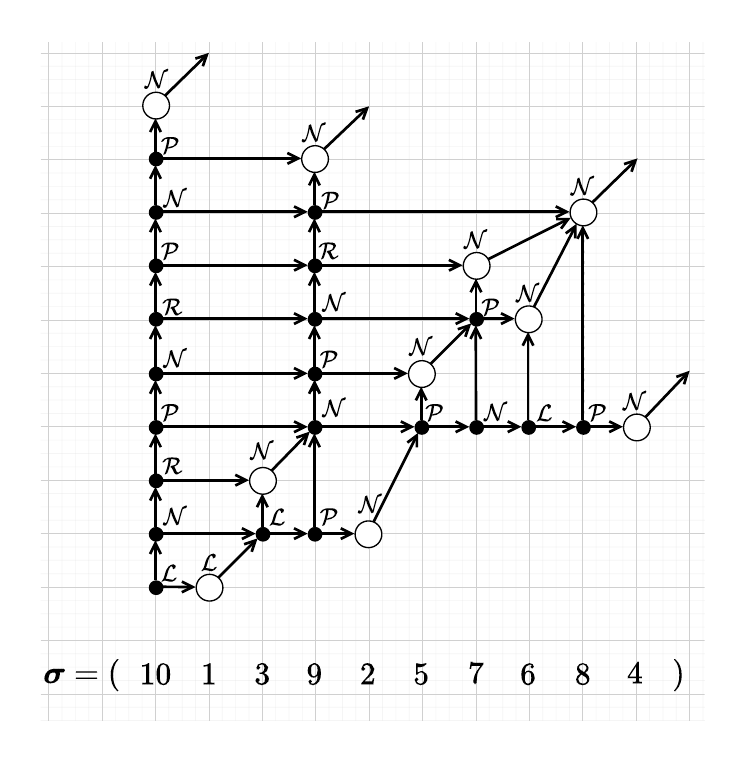}
\caption{The outcomes for the uniform case where $\bm{\sigma} = (10, 1, 3, 9, 2, 5, 7, 6, 8, 4)$ and $m \geq 2$}
\label{fig:uniform-outcome}
\end{figure}

\begin{example}
Let $m \geq 2, \bm{\sigma} = (10, 1, 3, 9, 2, 5, 7, 6, 8, 4)$, and $\bm{w} = (m, m, m, m, m, m, m, m, \allowbreak m, m) \in \wset^{(m)}$.
Then the outcomes $o^+(\spos{\bm{w}}xy) = o^-(\spos{\bm{w}}xy)$ for $\spos{\bm{w}}xy \in \sset{\bm{w}}$ are shown in Fig.~\ref{fig:uniform-outcome}.
Note that the outcomes does not depend on $m$ as long as $m \geq 2$, and thus the number $m$ of stones in the piles is omitted in the figure.
For example, $o(\bm{w}) = o(\spos{\bm{w}}11)$ obtained as
\begin{align}
o(\bm{w})
&= o(\spos{\bm{w}}11)\\
&\eqlab{A}= \phi_{\NN\to\RR} (o(\spos{\bm{w}}32))\\
&\eqlab{B}= \phi_{\NN\to\RR} \circ \phi_{\NN\to\LL}(o(\spos{\bm{w}}44)) \\
&\eqlab{C}= \phi_{\NN\to\RR} \circ \phi_{\NN\to\LL}( o(\spos{\bm{w}}65) ) \\
&\eqlab{D}= \phi_{\NN\to\RR} \circ \phi_{\NN\to\LL} \circ \phi_{\PP\to\NN}( o(\spos{\bm{w}}76) ) \\
&\eqlab{E}= \phi_{\NN\to\RR} \circ \phi_{\NN\to\LL} \circ \phi_{\PP\to\NN} \circ \phi_{\NN\to\PP} ( o(\spos{\bm{w}}98) ) \\
&\eqlab{F}= \phi_{\NN\to\RR} \circ \phi_{\NN\to\LL} \circ \phi_{\PP\to\NN} \circ \phi_{\NN\to\PP} \circ \phi_{\PP\to\NN} ( o(\bm{0}) ) \\
&= \phi_{\NN\to\RR} \circ \phi_{\NN\to\LL} \circ \phi_{\PP\to\NN} \circ \phi_{\NN\to\PP} \circ \phi_{\PP\to\NN} (\PP) \\
&= \phi_{\NN\to\RR} \circ \phi_{\NN\to\LL} \circ \phi_{\PP\to\NN} \circ \phi_{\NN\to\PP} (\NN) \\
&= \phi_{\NN\to\RR} \circ \phi_{\NN\to\LL} \circ \phi_{\PP\to\NN} (\PP) \\
&= \phi_{\NN\to\RR} \circ \phi_{\NN\to\LL} (\NN) \\
&= \phi_{\NN\to\RR} (\LL) \\
&= \LL,
\end{align}
where
(A) follows from the fourth case of Theorem \ref{thm:const-outcome},
(B) follows from the third case of Theorem \ref{thm:const-outcome},
(C) follows from the fifth case of Theorem \ref{thm:const-outcome},
(D) follows from the first case of Theorem \ref{thm:const-outcome},
(E) follows from the second case of Theorem \ref{thm:const-outcome},
and (F) follows from the first case of Theorem \ref{thm:const-outcome}.
\end{example}

\subsection{Proof of Theorem \ref{thm:const-outcome}}

We first have the following observation regarding $L(\cdot)$ and $R(\cdot)$ in the uniform case.

\begin{lemma}
\label{lem:const-const}
For any integer $m \geq 2$ and $\bm{w} \in \wset^{(m)}_+$, the following statements (i)--(iii) hold.
\begin{enumerate}[(i)]
\item $(L(\bm{w}_R), R(\bm{w}_L)) \not\in \{(0, 1), (1, 0)\}$.
\item $L(\bm{w})\not\in\{2, 3, \ldots, m-1\}$ and $R(\bm{w}) \not\in\{2, 3, \ldots, m-1\}$.
\item If $L(\bm{w}) = 1$ or $R(\bm{w}) = 1$, then $\bm{w} \in \wset_=$ and $\bm{w}_{\ast} \in \NN$.
\end{enumerate}
\end{lemma}

\begin{proof}[Proof of Lemma \ref{lem:const-const}]
We prove (i)--(iii) simultaneously by induction on $\nu(\bm{w})$.
In the case $\bm{w} \in \wset_=$, the statements (i)--(iii) follow directly from Theorem \ref{thm:outcome} (i).
Thus, we suppose
\begin{align}
\bm{w} \in \wset_{\neq}, \label{eq:r79kuln22hfq}
\end{align}
and prove the statements (i)--(iii) hold as follows.

(Proof of (i))
We assume
\begin{align}
\label{eq:ia6z6u2nziet}
L(\bm{w}_R) = 1
\end{align}
and claim $R(\bm{w}_L) \neq 0$.
We have
\begin{align}
\bm{w}_{LR} \eqlab{A}= \bm{w}_{RL} \eqlab{B}= \bm{w}_{R\ast} \eqlab{B}\in \NN, \label{eq:v4l60n9b9n28}
\end{align}
where
(A) follows from \eqref{eq:r79kuln22hfq} and Lemma \ref{lem:neq-lr} (i),
and (B)s follow because $\bm{w}_R \in \wset_=$ from \eqref{eq:ia6z6u2nziet} and the induction hypothesis for (iii).

By \eqref{eq:r79kuln22hfq}, we have $\bm{w}_L \in \wset_+$, so that one of $\bm{w}_L \in \wset_=$ and $\bm{w}_L \in \wset_{\neq}$ holds.
We consider these two cases separately as follows.
\begin{itemize}
\item The case $\bm{w}_L \in \wset_=$: We have $\bm{w}_{L\ast} = \bm{w}_{LR} \in \NN$ by \eqref{eq:v4l60n9b9n28}.
Hence, by Theorem \ref{thm:outcome} (i), we obtain $R(\bm{w}_L) = 1 \neq 0$ as desired.

\item The case $\bm{w}_L \in \wset_{\neq}$: 
We have
\begin{align}
L(\bm{w}_{LR})
\eqlab{A}< \alpha(\bm{w}_{LR})
\eqlab{B}= \alpha(\ppos{\bm{w}_L}{\cdot}{0})
\eqlab{C}= \alpha(\bm{w}_L),
 \label{eq:9pzzzltdm57n}
\end{align}
where
(A) follows from \eqref{eq:v4l60n9b9n28} and Lemma \ref{lem:thre} (a),
(B) follows from $\bm{w}_L \in \wset_{\neq}$ and Lemma \ref{lem:neq-lr} (ii),
and (C) follows from $\bm{w}_L \in \wset_{\neq}$.
Therefore,
\begin{align}
R(\bm{w}_L)
\eqlab{A}= R(\bm{w}_{LL}) + \alpha(\bm{w}_L) - L(\bm{w}_{LR})
\eqlab{B}> R(\bm{w}_{LL})
\geq 0,
\end{align}
where
(A) follows from \eqref{eq:9pzzzltdm57n} and Theorem \ref{thm:outcome} (ii) since $\bm{w}_L \in \wset_{\neq}$ by the assumption,
and (B) follows from \eqref{eq:9pzzzltdm57n}.
In particular, $R(\bm{w}_L) \neq 0$ as desired.
\end{itemize}
Therefore, we conclude that \eqref{eq:ia6z6u2nziet} implies $R(\bm{w}_L) \neq 0$.
Symmetrically, we see that $R(\bm{w}_L) = 1$ implies $L(\bm{w}_R) \neq 0$. This completes the proof of the statement (i).

(Proof of (ii))
By symmetry, we only show that $L(\bm{w}) = 0$ or $L(\bm{w}) \geq m$. 
If $m \leq R(\bm{w}_L)$, then $L(\bm{w}) = 0$ holds as desired directly from Theorem \ref{thm:outcome} (ii).
Thus, we assume
\begin{align}
\label{eq:glsuwgfg2iv1}
m > R(\bm{w}_L).
\end{align}
Then by \eqref{eq:r79kuln22hfq}, \eqref{eq:glsuwgfg2iv1} and Theorem \ref{thm:outcome} (ii), we have
\begin{align}
\label{eq:coxsgxzxdpa6}
L(\bm{w}) = L(\bm{w}_R) + m-R(\bm{w}_L).
\end{align}
We consider the following two cases.
\begin{itemize}
\item The case $L(\bm{w}_R) \geq m$: We have
\begin{align}
L(\bm{w})
\eqlab{A}= L(\bm{w}_R) + m - R(\bm{w}_L)
\eqlab{B}> L(\bm{w}_R)
\eqlab{C}\geq m
\end{align}
as desired, where
(A) follows from \eqref{eq:coxsgxzxdpa6},
(B) follows from \eqref{eq:glsuwgfg2iv1},
and (C) follows directly from the assumption.
\item The case $L(\bm{w}_R) < m$: Then by \eqref{eq:glsuwgfg2iv1}, 
we have $\{L(\bm{w}_R), R(\bm{w}_L)\} \subseteq \{0, 1, 2, \ldots, m-1\}$.
Moreover, the induction hypothesis for (ii) leads to $\{L(\bm{w}_R), R(\bm{w}_L)\} \subseteq \{0, 1\}$.
Hence, by the statement (i) shown above, it must hold
\begin{align}
L(\bm{w}_R) = R(\bm{w}_L) \in \{0, 1\}, \label{eq:mqwpbrkc0t1h}
\end{align}
so that
\begin{align}
L(\bm{w})
\eqlab{A}= L(\bm{w}_R) + m - R(\bm{w}_L)
\eqlab{B}= m
\end{align}
as desired, where
(A) follows from \eqref{eq:coxsgxzxdpa6},
and (B) follows from \eqref{eq:mqwpbrkc0t1h}.
\end{itemize}

(Proof of (iii))
By the statement (ii) shown above, we have $L(\bm{w}) = 0$, $L(\bm{w}) = 1$, or $L(\bm{w}) \geq m$.
The same applies to $R(\bm{w})$.
Also, by the statement (i) shown above,
neither $(L(\bm{w}), R(\bm{w})) = (0, 1)$ nor $(L(\bm{w}), R(\bm{w})) = (1, 0)$ is possible.
Table \ref{tab:const-thre} indicates the values $L(\bm{w})$ obtained by \eqref{eq:r79kuln22hfq} and Theorem \ref{thm:outcome} (ii) for all possible pairs $(L(\bm{w}_R), R(\bm{w}_L))$.
The table shows that neither $L(\bm{w}) = 1$ nor $R(\bm{w}) = 1$ occurs, and thus the statement (iii) vacuously holds.
\end{proof}

\begin{table}
\centering
\caption{The values of $L(\bm{w})$ for all possible pairs $(L(\bm{w}_R), R(\bm{w}_L))$,
where $\geq m$ denotes the integers greater than or equal to $m$.}
\begin{tabular}{c|c|c|c}
\label{tab:const-thre}
   \diagbox{$L(\bm{w}_R)$}{$R(\bm{w}_L)$} & $0$ & $1$ & $\geq m$ \\
   \hline
   $0$ & $m$ & - & $0$ \\
   \hline
   $1$ & - & $m$ & $0$ \\
   \hline
   $\geq m$ & $\geq m$ & $\geq m$ & $0$
\end{tabular}
\end{table}

The following lemma claims that $o(\bm{w})$ is determined only by
the outcome of the subpositions $\bm{w}_L$ and $\bm{w}_R$ regardless of the value of $m$.

\begin{lemma}
\label{lem:const-thre}
For any integer $m \geq 2$, the following statements (i) and (ii) hold.
\begin{enumerate}[(i)]
\item For any $\bm{w} \in \wset_=^{(m)}$, the following equivalences (a) and (b) hold.
\begin{enumerate}[(a)]
\item $\bm{w} \in \LL \cup \PP \iff \bm{w}_{\ast} \in \LL$.
\item $\bm{w} \in \RR \cup \PP \iff \bm{w}_{\ast} \in \RR$.
\end{enumerate}
\item For any $\bm{w} \in \wset_{\neq}^{(m)}$, the following equivalences (a) and (b) hold.
\begin{enumerate}
\item $\bm{w} \in \LL \cup \PP \iff \bm{w}_R \in \LL \cup \NN$.
\item $\bm{w} \in \RR \cup \PP \iff \bm{w}_L \in \RR \cup \NN$.
\end{enumerate}
\end{enumerate}
\end{lemma}

\begin{proof}[Proof of Lemma \ref{lem:const-thre}]
(Proof of (i))
The statement (a) is shown as
\begin{align}
\bm{w} \in \LL \cup \PP
\eqlab{A}\iff R(\bm{w}) \geq m
\eqlab{B}\iff \bm{w}_{\ast} \in \LL,
\end{align}
where
(A) follows from Lemma \ref{lem:thre} (b),
and (B) follows from Theorem \ref{thm:outcome} (i) and $m \geq 2$.
The statement (b) is shown symmetrically.

(Proof of (ii))
The statement (b) is shown as
\begin{align}
\bm{w} \in \RR \cup \PP
\eqlab{A}\iff L(\bm{w}) \geq m
\eqlab{B}\iff R(\bm{w}_L) < m
\eqlab{C}\iff \bm{w}_L \in \RR \cup \NN,
\end{align}
where
(A) follows from Lemma \ref{lem:thre} (a),
(B) follows from the discussion in the proof of Lemma \ref{lem:const-const} (iii) and Table \ref{tab:const-thre},
and (C) follows from Lemma \ref{lem:thre} (b).
The statement (a) is shown symmetrically.
\end{proof}

By the above lemma, the desired theorem is proved as follows.

\begin{proof}[Proof of Theorem \ref{thm:const-outcome}]
The case $\bm{w} \in \wset_=$ is shown in Table \ref{tab:eq}. The first column $o(\bm{w}_{\ast})$ ranges over all the outcome classes.
The second and third columns are obtained from the first column by Lemma \ref{lem:const-thre} (i).
The fourth column $o(\bm{w})$ is determined by the second and third columns.
Comparing the first and sixth columns, we see $o(\bm{w}) = \phi_{\PP \to \NN}(\bm{w}_{\ast})$ as desired.

The case $\bm{w} \in \wset_{\neq}, \bm{w}_L \in \wset_=, \bm{w}_R \in \wset_{\neq}$ is
shown in Table \ref{tab:neq-eqneq}. The first column $o(\bm{w}_{LR})$ ranges over all the outcome classes.
The second column is obtained from the first column by the assertion for the case $\bm{w} \in \wset_=$ shown above.
The third column is obtained from the first column by Lemma \ref{lem:const-thre} (ii) (b).
The fourth (resp.~fifth) column is obtained from the third (resp.~second) column by Lemma \ref{lem:const-thre} (ii) (a) (resp.~Lemma \ref{lem:const-thre} (ii) (b)).
The sixth column is determined by the fourth and fifth columns.
The case $\bm{w} \in \wset_{\neq}, \bm{w}_L \in \wset_{\neq}, \bm{w}_R \in \wset_{=}$ is shown by the symmetric argument.

The other cases are similarly shown in Table \ref{tab:neq-eqeq} and Table \ref{tab:neq-neqneq}.
\end{proof}

\begin{table}[htbp]
\centering
\caption{The case $\bm{w} \in \wset_=$.}
\begin{tabular}{c|c|c|c}
\label{tab:eq}
   $o(\bm{w}_{\ast})$ & $\bm{w} \in \LL \cup \PP$ & $\bm{w} \in \RR \cup \PP$ & $o(\bm{w})$ \\
   \hline
   $\LL$ & true & false & $\LL$\\
   $\RR$ & false & true & $\RR$\\
   $\PP$ & false & false & $\NN$\\
   $\NN$ & false & false & $\NN$\\
\end{tabular}
 
\caption{The case $\bm{w} \in \wset_{\neq}, \bm{w}_L \in \wset_=, \bm{w}_R \in \wset_=$.}
\begin{tabular}{c|c|c|c|c|c}
\label{tab:neq-eqeq}
   $o(\bm{w}_{LR})$ & $o(\bm{w}_L)$ & $o(\bm{w}_R)$ & $\bm{w} \in \LL \cup \PP$ & $\bm{w} \in \RR \cup \PP$ & $o(\bm{w})$ \\
   \hline
   $\LL$ & $\LL$ & $\LL$ & true & false & $\LL$\\
   $\RR$ & $\RR$ & $\RR$ & false & true & $\RR$\\
   $\PP$ & $\NN$ & $\NN$ & true & true & $\PP$\\
   $\NN$ & $\NN$ & $\NN$ & true & true & $\PP$\\
\end{tabular}

\caption{The case $\bm{w} \in \wset_{\neq}, \bm{w}_L \in \wset_=, \bm{w}_R \in \wset_{\neq}$.}
\begin{tabular}{c|c|c|c|c|c}
\label{tab:neq-eqneq}
   $o(\bm{w}_{LR})$ & $o(\bm{w}_L)$ & $o(\bm{w}_R)$ & $\bm{w} \in \LL \cup \PP$ & $\bm{w} \in \RR \cup \PP$ & $o(\bm{w})$ \\
   \hline
   $\LL$ & $\LL$ & $\LL$ or $\NN$ & true & false & $\LL$\\
   $\RR$ & $\RR$ & $\RR$ or $\PP$ & false & true & $\RR$\\
   $\PP$ & $\NN$ & $\LL$ or $\NN$ & true & true & $\PP$\\
   $\NN$ & $\NN$ & $\RR$ or $\PP$ & false & true & $\RR$\\
\end{tabular}

\caption{The case $\bm{w} \in \wset_{\neq}, \bm{w}_L \in \wset_{\neq}, \bm{w}_R \in \wset_{\neq}$.}
\begin{tabular}{c|c|c|c|c|c}
\label{tab:neq-neqneq}
   $o(\bm{w}_{LR})$ & $o(\bm{w}_L)$ & $o(\bm{w}_R)$ & $\bm{w} \in \LL \cup \PP$ & $\bm{w} \in \RR \cup \PP$ & $o(\bm{w})$ \\
   \hline
   $\LL$ & $\LL$ or $\PP$ & $\LL$ or $\NN$ & true & false & $\LL$\\
   $\RR$ & $\RR$ or $\NN$ & $\RR$ or $\PP$ & false & true & $\RR$\\
   $\PP$ & $\RR$ or $\NN$ & $\LL$ or $\NN$ & true & true & $\PP$\\
   $\NN$ & $\LL$ or $\PP$ & $\RR$ or $\PP$ & false & false & $\NN$\\
\end{tabular}
\end{table}

\section{The Integrality of Atomic Weights}
\label{sec:aw}

The exact game values for each position are too complicated to analyze, even when limited to \textsc{Partizan-End-Nim}, as suggested in \cite{AN01}.
On the other hand, since \textsc{Partizan-Serial-Nim} is dicotic, it is useful to consider atomic weights.
In this section, we prove that the atomic weight of a position in \textsc{Partizan-Serial-Nim} is an integer, while the atomic weight of a game is a game in general.
For the definition and details of atomic weights of games, we refer to textbooks \cite{ANW19,BCG18,Con00,Sie13}.

The atomic weight $\aw(G)$ of a game $G$ is defined not for all games $G$.
If $\aw(G)$ is defined, then $G$ is said to be \emph{atomic}.
It is known that the atomic weight $\aw(G)$ of a dicotic game $G$ is defined and calculated as follows.

\begin{proposition}[{\cite[Definition II.7.14, Theorem II.7.15]{Sie13}}]
\label{prop:aw-calc}
Let $G$ be an arbitrary dicotic game.
We define a game $\tilde{v}(G)$ recursively as 
\begin{align}
\tilde{v}(G) = \{\aw(G^L) - 2 \mid \aw(G^R) + 2\},
\end{align}
where $G^L$ (resp.~$G^R$) ranges over all Left (resp.~Right) options of $G$, respectively.
If $\tilde{v}(G)$ is not equal to an integer, then $\aw(G) = \tilde{v}(G)$.
If $\tilde{v}(G)$ is equal to an integer, then
\begin{align}
\aw(G) =
\begin{cases}
\text{the smallest element of}\,\, \mathcal{I} &\,\,\text{if}\,\, G < \cgfarstar,\\
\text{the largest element of}\,\, \mathcal{I} &\,\,\text{if}\,\, G > \cgfarstar,\\
0 &\,\,\text{if}\,\, G \cgfuzzy \cgfarstar, \label{eq:jo89m0gtjc7i}
\end{cases}
\end{align}
where
\begin{align}
\mathcal{I} = \{n \in \mathbb{Z} : \aw(G^L) - 2 \cglfuz n \cglfuz \aw(G^R)+2 \,\,\text{for all}\,\, G^L \,\,\text{and}\,\, G^R\},
\end{align}
and $\cgfarstar$ denotes the \emph{far star}, i.e., a nimber $\st m$ for a sufficiently large $m \in \pint$ such that $\st m$ is not a subposition of $G$.
\end{proposition}

The following fundamental properties of atomic weights are known.

\begin{proposition}[{\cite[Theorem II.7.12]{Sie13}}]
\label{prop:aw-prop}
For any atomic games $G$ and $H$, the following statements (i) and (ii) hold.
\begin{enumerate}[(i)]
\item $G+H$ is atomic and $\aw(G + H) = \aw(G) + \aw(H)$.
\item If $\aw(G) \geq \aw(H)$, then $o(G + \cgfarstar) \geq o(H + \cgfarstar)$.
\end{enumerate}
\end{proposition}

The atomic weights of typical games are given as follows.

\begin{proposition}[{\cite[Section II.7]{Sie13}}]~
\label{prop:aw-value}
\begin{enumerate}[(i)]
\item For any $n \in \mathbb{Z}$, we have $\aw(n \cdot \up) = n$.
\item For any $n \in \uint$, we have $\aw(\st n) = 0$.
\end{enumerate}
\end{proposition}

Our main theorem is as follows.

\begin{theorem}
\label{thm:aw-integer}
For any $\bm{w} \in \wset$, we have $\aw(\bm{w}) \in \mathbb{Z}$.
\end{theorem}

\begin{remark}
While outcomes in the uniform case with $m \geq 2$ do not depend on $m$,
atomic weights may depend on $m$ 
because of the following counterexample: $\aw((2, 2, 2, 2)) = 0 \neq 1 = \aw((3, 3, 3, 3))$, where $n = 4$ and $\bm{\sigma} = (3, 2, 4, 1)$.
\end{remark}

\subsection{Proof of Theorem \ref{thm:aw-integer}}
The proof relies on the following lemma.

\begin{lemma}
\label{lem:aw-option}
For any non-zero dicotic game $G$, if $\aw(G)$ is an integer and the atomic weights of all options of $G$ are integers, then the following statements (i) and (ii) hold.
\begin{enumerate}[(i)]
\item For any Left option $G^L$ of $G$, we have $\aw(G) \geq \aw(G^L) - 1$.
\item For any Right option $G^R$ of $G$, we have $\aw(G) \leq \aw(G^R) + 1$.
\end{enumerate}
\end{lemma}

\begin{proof}[Proof of Lemma \ref{lem:aw-option}]
We have
\begin{align}
\mathcal{I}
&= \{n \in \mathbb{Z} : \aw(G^L) - 2 \cglfuz n \cglfuz \aw(G^R) + 2 \,\,\text{for all}\,\, G^L \,\,\text{and}\,\, G^R\}\\
&\eqlab{A}= \{n \in \mathbb{Z} : \aw(G^L) - 1 \leq n \leq \aw(G^R) + 1 \,\,\text{for all}\,\, G^L \,\,\text{and}\,\, G^R\}, \label{eq:jmvly1i3wrzr}
\end{align}
where
(A) follows from the integrality of $\aw(G^L)$ and $\aw(G^R)$.
Since $\aw(G)$ is an integer, the game $\tilde{v}(G)$ in Proposition \ref{prop:aw-calc} must be an integer.
Hence, the above set $\mathcal{I}$ is not empty, so that
there exists $n \in \mathbb{Z}$
such that for any Left option $G^L$ and Right option $G^R$, it holds that $\aw(G^L) - 1 \leq n \leq \aw(G^R) + 1$.
In particular, for any Left option $G^L$ and Right option $G^R$, we have
\begin{align}
\aw(G^L) - 1 \leq \aw(G^R) + 1. \label{eq:ayxtbjubp0i5}
\end{align}

We claim $\aw(G) \in \mathcal{I}$, so that $\aw(G^L) - 1 \leq \aw(G) \leq \aw(G^R) + 1$ as desired.
By \eqref{eq:jo89m0gtjc7i}, if $G < \cgfarstar$ or $G > \cgfarstar$, then $\aw(G) \in \mathcal{I}$ clearly holds.
Thus, we need to show that $G \cgfuzzy \cgfarstar$ implies $\mathcal{I} \owns 0$.

To prove the contraposition, we assume $\mathcal{I} \not\owns 0$.
Without loss of generality, we may assume that $0 < \aw(G^L) - 1$ holds for some Left option $G^L$; that is,
\begin{align}
\aw(G^L) \geq 2. \label{eq:9j2tjbq8x3id}
\end{align}
for some Left option $G^L$.
For all Right options $G^R$, we have
\begin{align}
\aw(G^R)
\eqlab{A}\geq \aw(G^L)-2
\eqlab{B}\geq 0, \label{eq:v1tizhnkis5r}
\end{align}
where 
(A) follows from \eqref{eq:ayxtbjubp0i5},
and (B) follows from \eqref{eq:9j2tjbq8x3id}.

To complete the proof, it suffices to show $G > \cgfarstar$ (i.e., $G + \cgfarstar > 0$).
Left can win $G + \cgfarstar$ playing first by moving to $G^L + \cgfarstar$, which satisfies
\begin{align}
o(G^L + \cgfarstar) \eqlab{A}\geq o(\doubleup + \cgfarstar) = \LL,
\end{align}
where
(A) follows from Proposition \ref{prop:aw-prop} (ii) since $\aw(G^L) \geq 2 = \aw(\doubleup)$ by \eqref{eq:9j2tjbq8x3id} and Proposition \ref{prop:aw-value} (i).
It is confirmed that Left can win $G + \cgfarstar$ playing second as follows.
\begin{itemize}
\item Any Right option $G^R + \cgfarstar$ is in $\LL \cup \NN$ because
\begin{align}
o(G^R + \cgfarstar) \eqlab{A}\geq o(\cgfarstar) = \NN,
\end{align}
where
(A) follows from Proposition \ref{prop:aw-prop} (ii) since $\aw(G^R) \geq 0 = \aw(0)$ by \eqref{eq:v1tizhnkis5r}.
\item Any Right option $G + \st a$ with some $a \in \uint$ is reverted to $G^L + \st a$, which is in $\LL$ by the two-ahead rule \cite[Theorem II.7.13]{Sie13} because
\begin{align}
\aw(G^L + \st a)
\eqlab{A}= \aw(G^L) + \aw(\st a)
\eqlab{B}\geq 2 + 0 = 2,
\end{align}
where
(A) follows from Proposition \ref{prop:aw-prop} (i),
and (B) follows from \eqref{eq:9j2tjbq8x3id} and Proposition \ref{prop:aw-value} (ii).
\end{itemize}
Hence, we conclude that $G > \cgfarstar$.
\end{proof}

Using the above lemma, we prove the integrality of atomic weights as follows.

\begin{proof}[Proof of Theorem \ref{thm:aw-integer}]
We prove this by induction on $\nu(\bm{w})$.
For the base case $\nu(\bm{w}) = 0$, clearly we have $\aw(\bm{0}) = 0$.
We consider the induction step for $\bm{w} \in \wset_+$ below.

We first consider the case $\bm{w} \in \wset_=$.
We show that for any $\gamma \in \uint$, it holds that $\aw(\pos{\bm{w}}{\gamma}) \in \mathbb{Z}$ by induction on $\gamma$.
For the base case $\gamma = 0$, we have $\aw(\pos{\bm{w}}{0}) = \aw(\bm{w}_{\ast}) \in \mathbb{Z}$ directly from the induction hypothesis on $\nu(\bm{w})$.
We consider the induction step $\gamma \geq 1$.
By the induction hypothesis on $\gamma$, the atomic weights of all options of $\pos{\bm{w}}{\gamma}$ are integers,
so that the following maximum and minimum are defined as integers:
\begin{align}
M &\coloneqq \max_{0 \leq \gamma' < \gamma} \aw(\pos{\bm{w}}{\gamma'})
\eqqcolon \aw(\pos{\bm{w}}{\hat{\gamma}}),\\
m &\coloneqq \min_{0 \leq \gamma' < \gamma} \aw(\pos{\bm{w}}{\gamma'})
\eqqcolon \aw(\pos{\bm{w}}{\check{\gamma}}).
\end{align}
Then
\begin{align}
M = \aw(\pos{\bm{w}}{\hat{\gamma}})
\eqlab{A}\leq \aw(\pos{\bm{w}}{\check{\gamma}})+1
= m + 1,
\end{align}
where (A) follows from Lemma \ref{lem:aw-option} (i) and (ii) and the induction hypothesis on $\gamma$; for example, if $\hat{\gamma} < \check{\gamma}$, then 
$\pos{\bm{w}}{\hat{\gamma}}$ is a Left option of $\pos{\bm{w}}{\check{\gamma}}$, and thus
(A) is seen by Lemma \ref{lem:aw-option} (i).
This yields $M-2 \cglfuz m \cglfuz m+2$, so that $\tilde{v}(\pos{\bm{w}}{\gamma})$ in Proposition \ref{prop:aw-calc} is an integer.
Therefore, $\aw(\pos{\bm{w}}{\gamma}) \in \mathbb{Z}$.

We next consider the case $\bm{w} \in \wset_{\neq}$.
We show that $\aw(\ppos{\bm{w}}{\alpha}{\beta}) \in \mathbb{Z}$ by induction on $\alpha + \beta$.
If $\alpha = 0$ or $\beta = 0$, then $\aw(\ppos{\bm{w}}{\alpha}{\beta}) \in \mathbb{Z}$ directly from the induction hypothesis on $\nu(\bm{w})$.
Suppose $\alpha \geq 1$ and $\beta \geq 1$.
By the induction hypothesis on $\alpha + \beta$, all options of $\ppos{\bm{w}}{\alpha}{\beta}$ are integers.
Let
\begin{align}
M &\coloneqq \max_{0 \leq \alpha' < \alpha} \aw(\ppos{\bm{w}}{\alpha'}{\beta})
\eqqcolon \aw(\ppos{\bm{w}}{\hat{\alpha}}{\beta}),\\
m &\coloneqq \min_{0 \leq \beta' < \beta} \aw(\ppos{\bm{w}}{\alpha}{\beta'})
\eqqcolon \aw(\ppos{\bm{w}}{\alpha}{\check{\beta}}).
\end{align}
Then
\begin{align}
M = \aw(\ppos{\bm{w}}{\hat{\alpha}}{\beta})
\eqlab{A}\leq \aw(\ppos{\bm{w}}{\hat{\alpha}}{\check{\beta}})+1
\eqlab{B}\leq \aw(\ppos{\bm{w}}{\alpha}{\check{\beta}})+2
= m + 2,
\end{align}
where
(A) follows from Lemma \ref{lem:aw-option} (ii) and $\check{\beta} < \beta$,
(B) follows from Lemma \ref{lem:aw-option} (i) and $\hat{\alpha} < \alpha$.
This yields $M-2 \cglfuz m+1 \cglfuz m+2$, so that $\tilde{v}(\ppos{\bm{w}}{\alpha}{\beta})$ in Proposition \ref{prop:aw-calc} is an integer.
Therefore, $\aw(\ppos{\bm{w}}{\alpha}{\beta}) \in \mathbb{Z}$.
\end{proof}

\bibliographystyle{plain}
\bibliography{draft}

@article{Bou1901,
	ISSN = {0003486X, 19398980},
	author = {C. L. Bouton},
	journal = {Annals of Mathematics},
	number = {1/4},
	pages = {35--39},
	publisher = {[Annals of Mathematics, Trustees of Princeton University on Behalf of the Annals of Mathematics, Mathematics Department, Princeton University]},
	title = {Nim, A Game with a Complete Mathematical Theory},
	urldate = {2025-11-14},
	volume = {3},
	year = {1901}
}

@ARTICLE{SU93,
	author  = {W. Stromquist and D. Ullman},
	title   = {Sequential compounds of combinatorial games},
	journal = {Theoretical Computer Science}, 
	volume  = {119},
	year    = {1993},
	pages   = {311--321}
}

@ARTICLE{Ste07,
	author  = {F. Stewart},
	title   = {The sequential join of combinatorial games},
	journal = {Integers}, 
	volume  = {7},
	year    = {2007},
	pages	= {G03}
}

@article{Has26,
	title = {The game value of sequential compounds of integers and stars},
	journal = {Theoretical Computer Science},
	volume = {1080},
	pages = {116043},
	year = {2026},
	issn = {0304-3975},
	author = {K. Hashimoto},
}

@article{AN01,
	author = {Albert, M. and Nowakowski, R.},
	year = {2001},
	month = {02},
	title = {The Game of End-Nim},
	volume = {8},
	number = {2},
	pages = {R1},
	journal = {The Electronic Journal of Combinatorics},
}

@article{Lev06,
	author       = {L. Levine},
	title        = {Fractal Sequences and Restricted Nim},
	journal      = {Ars Combinatoria},
	volume       = {80},
	year         = {2006},
	pages        = {113--127},
	bibsource    = {dblp computer science bibliography, https://dblp.org}
}

@article{CN11,
	author = {Cairns, G. and Ho, N. B.},
	title = {Some Remarks on End-Nim},
	journal = {International Journal of Combinatorics},
	volume = {2011},
	number = {1},
	pages = {824742},
	year = {2011}
}

@article{LW19,
	title = {Multi-player End-Nim games},
	journal = {Theoretical Computer Science},
	volume = {761},
	pages = {7--22},
	year = {2019},
	issn = {0304-3975},
	author = {W. A. Liu and T. Wu},
}

@Incollection{Guy95,
	author = {Guy, R. K.},
	title = {Unsolved Problems in Combinatorial Games},
	booktitle = {Combinatorics Advances},
	year = {1995},
	publisher = {Springer US},
	address = {Boston, MA},
	pages = {161--179},
	isbn = {978-1-4613-3554-2}
}

@Incollection{DKW09,
	place={Cambridge},
	title={Ordinal partizan End Nim}, 
	booktitle={Games of No Chance 3},
	publisher={Cambridge University Press},
	author={Duffy, A. and Kolpin, G. and Wolfe, D.},
	year={2009},
	pages={419--426}
}

@Incollection{FR09, 
	place={Cambridge},
	title={The game of End-Wythoff},
	booktitle={Games of No Chance 3},
	publisher={Cambridge University Press},
	author={Fraenkel, A. S. and Reisner, E.},
	year={2009},
	pages={329--348}
}

@book{Sie13,
	title={Combinatorial Game Theory},
	author={Siegel, A. N.},
	isbn={9780821851906},
	lccn={2012043675},
	series={Graduate studies in mathematics},
	year={2013},
	publisher={American Mathematical Society}
}

@book{BCG18,
	title={Winning Ways for Your Mathematical Plays: Volume 1},
	author={Berlekamp, E. R. and Conway, J. H. and Guy, R. K.},
	isbn={9780429945595},
	year={2018},
	publisher={CRC Press}
}

@book{ANW19,
	title={Lessons in Play: An Introduction to Combinatorial Game Theory, Second Edition},
	author={Albert, M. and Nowakowski, R. and Wolfe, D.},
	isbn={9780429524097},
	lccn={2020693737},
	year={2019},
	publisher={CRC Press}
}

@book{Con00,
	title={On Numbers and Games},
	author={Conway, J. H.},
	isbn={9781568811277},
	lccn={00046927},
	series={Ak Peters Series},
	year={2000},
	publisher={Taylor \& Francis}
}

@book{HG16,
	title={An Introduction to Combinatorial Game Theory},
	author={Haff, L. R. and Garner, W. J.},
	isbn={9781365973826},
	year={2016},
	publisher={Lulu.com}
}

\end{document}